\documentclass{birkjour}

\usepackage{url}
\usepackage{bm}
\usepackage{caption}
\allowdisplaybreaks
 \newtheorem{theorem}{Theorem}[section]
 \newtheorem{corollary}[theorem]{Corollary}
 \newtheorem{lemma}[theorem]{Lemma}
 \newtheorem{proposition}[theorem]{Proposition}
 \theoremstyle{definition}
 
 \theoremstyle{remark}
 \newtheorem{remark}[theorem]{Remark}
 
 \numberwithin{equation}{section}
 
 \newcommand{\R}{\mathbb{R}}
\newcommand{\N}{\mathbb{N}}

\newcommand{\ra}{\rangle}
\newcommand{\la}{\langle}
\newcommand{\bnu}{{\bm \nu}}
\newcommand{\btau}{{\bm \tau}}
\newcommand{\bv}{{\bm v}}
\newcommand{\bkappa}{{\bm \kappa}}
\newcommand{\gm}{\gamma}
\DeclareMathOperator*{\esssup}{ess\,sup}

\begin{document}

%
%
%
%
%
%
%
%
%

\title[ACF for open curves with general contact angles on skew lines]
 {Area-preserving curvature flow\\ for open curves with general contact angles\\ on skew lines}

\author[Fuya Hiroi]{Fuya Hiroi}

\address{%
Graduate School of Mathematical Sciences\\
The University of Tokyo\\
3-8-1 Komaba, Meguro-ku\\
Tokyo 153-8914\\
Japan}

\email{fuya-hiroi@g.ecc.u-tokyo.ac.jp}

\thanks{The work was supported by JST SPRING Grant Number JPMJSP2114. The author was partially supported by KAKENHI Grant Number 22H05105.}


\date{October 1, 2026}

\begin{abstract}
We study an area-preserving curvature flow for planar open curves whose end points lie on skew lines and meet the lines at general contact angles. In previous studies of moving boundary problems for higher-order geometric gradient flows, the right-angle condition has often been imposed to ensure that the evolving curves naturally satisfy the boundary conditions. However, under general contact-angle conditions, the evolution must be described by a free boundary problem. Under suitable assumptions on the contact angles $\theta_\pm$, the angle $\theta$ between the skew lines, and the initial data $\gamma_0$, we prove global-in-time existence and uniqueness up to reparametrization. We further prove that, after choosing a constant-speed parametrization at each time, the solution converges exponentially to a circular arc with the same signed area as the initial curve. Our analysis is based on a formulation of the evolution as a free boundary problem together with the derivation of uniform a priori estimates.
\end{abstract}

\maketitle
\section{Introduction}\label{sec1}

In this paper, we consider the area-preserving curvature flow for planar open curves $\gamma : [-1, 1] \times [0, T) \to \R^2$ whose end points lie on two skew lines and meet them at prescribed contact angles. For $m \in \N \cup \{ 0 \}$, the $m$-th area-preserving curvature flow is defined by
  \begin{align}\label{eq:ACF}\tag{$m$-ACF}
     \partial_t \gamma &= (-1)^m \partial_s^{2m} (k - \bar{k}) \bnu,
  \end{align}
where $s$, $k$, $\bar{k}: = \frac{1}{L[\gamma]}\int_\gamma k \, ds$, and $\bnu$ denote, respectively, the arc length parameter of $\gamma$, the curvature of $\gamma$, the integral mean of $k$ with respect to $s$, and the unit inner normal vector defined by rotating the unit tangential vector $\btau := \partial_s \gamma$. We use the Frenet--Serret formulas:
  \begin{align}\label{FS}
    \partial_s^2 \gamma = \partial_s \btau = k \bnu, \qquad \partial_s \bnu = -k \btau.
  \end{align}
For simplicity, we define 
  \[F_m[k] := (-1)^m \partial_s^{2m}(k - \bar{k}). \]
When $\gamma$ is a closed curve, \eqref{eq:ACF} can be regarded as the $H^{-m}$-gradient flow for the length defined by
  \[L[\gamma] := \int_{-1}^1 | \partial_u \gamma (u) | \, du = \int_\gamma \, ds, \]
while preserving the signed area defined by
  \[A[\gamma] := - \frac{1}{2} \int_\gamma \la \gamma, \bnu \ra \, ds,\]
where $\la \cdot, \cdot \ra: \R^2 \times \R^2 \to \R$ denotes the Euclidean inner product.
For $m \geq 2$, \eqref{eq:ACF} appears as polyharmonic heat flow in Parkins--Wheeler~\cite{PW_2016}. More recently, Miura~\cite{M_2024} studied immortal solutions of \eqref{eq:ACF} in the setting of closed curves.
We focus on the moving boundary problem for \eqref{eq:ACF} in the setting of planar open curves with end points on skew lines. For $\alpha \in [0, 2\pi)$, define the line
  \begin{align*}
    \eta_\alpha := \left\{ \lambda \bv_\alpha \middle| \lambda \in \R \right\},
  \end{align*}
where ${\bm v}_\alpha := (\cos \alpha, \sin \alpha)$ and we consider open curves $\gamma$ satisfying $\gamma (-1, t) \in \eta_0$ and $\gamma (1, t) \in \eta_\theta$. We note that each end point remains on its designated line. In this sense, we refer to those two lines $\eta_0$ and $\eta_\theta$ as skew lines. Here we also note that $A[\gamma]$ is well defined for open curves with the end points on skew lines. Indeed, $A[\gamma]$ denotes the signed area enclosed by the curve and the line segments joining its end points to the origin. Furthermore, when the curve is parametrized in the counterclockwise direction, we define $\theta_-$ as the angle between $\btau (-1, t)$ and ${\bm v}_0$, and $\theta_+$ as the angle between $\btau (1, t)$ and ${\bm v}_\theta$ (see Fig. \ref{fig1}).
\begin{figure}[htbp]
  \centering
  \includegraphics[width=0.42\textwidth]{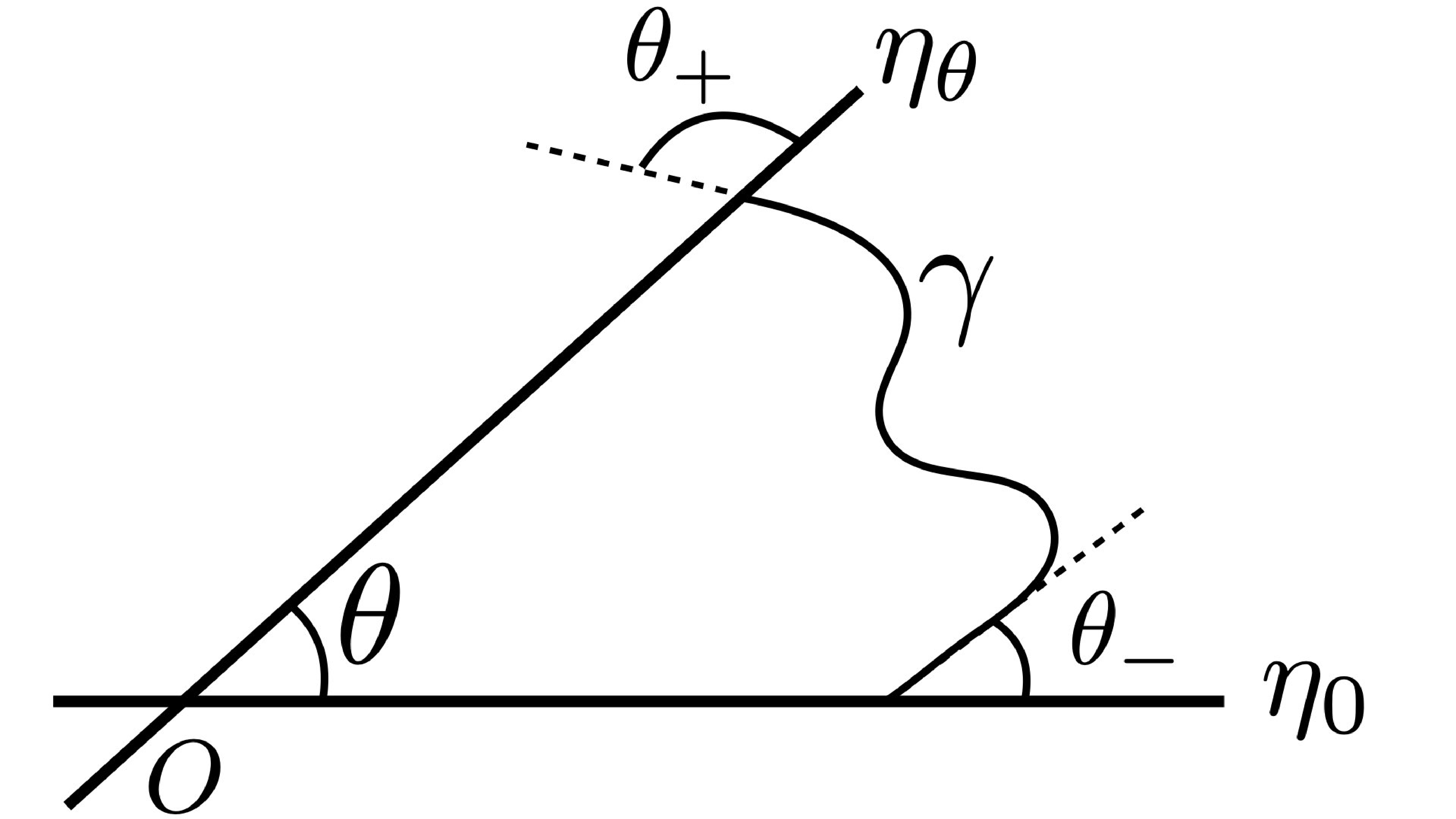}
  \caption{Open curves with the end points on skew lines}
  \label{fig1}
\end{figure}

\noindent
In the following, a curve $\gamma : [-1, 1] \to \R^2$ is called regular if $\gamma (u)$ satisfies 
\[
\gamma \in C^1 ([-1, 1]; \R^2), \quad \text{and} \quad 
| \gamma' (u)| > 0 \quad \text{for all} \quad u \in [-1, 1]. 
\]
%

Recently, the framework of open curves with end points on skew lines has emerged as an effective setting for studying the asymptotic behavior of area-preserving geometric gradient flows. A representative example is the area-preserving elastic flow, namely, the $H^{-1}$-gradient flow of the bending energy. 
Elastic flows with either a length constraint or a length penalization have been extensively studied for both closed and open curves; see, e.g.,
\cite{BHV_2022, DKS_2002, DLP_2017, DP_2014, KM_2024, K_1992, L_2012, L_1991, MPP_2021, MP_2021, M_2025, MMR_2025, MR_2026, NO_2014, NO_2017, O_2008, O_2007, OPW_2020, OS_2025, OW_2023, P_1996, P_2022, RS_2024, W_1995, W_2015}
and the references therein.
By contrast, the elastic flow subject only to an area constraint is substantially more delicate. Okabe~\cite{O_2007} studied the elastic flow under an area constraint together with an inextensibility condition, whereas Langer~\cite{L_2025} introduced the area-preserving elastic flow as the $H^{-1}$-gradient flow of the bending energy. For closed curves, Langer~\cite{L_2025} proved the global-in-time existence of smooth solutions and established compactness when an additional length penalization is imposed. In the absence of such a penalization, however, compactness for the area-preserving elastic flow of closed curves remains an open problem.
The skew-line setting provides a new perspective on this problem. In our joint work with Okabe~\cite{HO_2026}, we studied the area-preserving elastic flow for open curves whose end points lie on skew lines and satisfy the right-angle condition. Under suitable assumptions on the angle $\theta$ between the two lines and on the initial curve, we established the compactness of the flow without imposing either a length constraint or a length penalization. This result demonstrates that appropriate boundary conditions can provide sufficient geometric control to ensure compactness. Accordingly, the skew-line setting has emerged as a natural framework for studying the long-time behavior of area-preserving geometric gradient flows.
Another representative example is the area-preserving curvature flow \eqref{eq:ACF}. For $m=1$, this flow reduces to the curve diffusion flow. The curve diffusion flow has attracted considerable attention as a fundamental higher-order geometric evolution equation that may develop singularities. In particular, its long-time behavior and compactness for open curves have been investigated by adapting ideas and estimates developed for closed curves; see, e.g., \cite{GM_2025, HO_2025, WW_2024}. In our previous work with Okabe~\cite{HO_2025}, we established the compactness of the curve diffusion flow for open curves with end points on skew lines satisfying the right-angle condition under suitable assumptions on $\theta$  and on the initial curve.

These results naturally raise two questions: whether the right-angle condition can be relaxed and whether the analytical framework can be extended to more general higher-order geometric gradient flows. To answer these questions, we take the area-preserving curvature flow \eqref{eq:ACF} as a model problem and develop an analytical framework under general contact-angle conditions, with the aim of extending it to a broader class of higher-order geometric gradient flows. Because its analytical structure is simpler than that of the area-preserving elastic flow, \eqref{eq:ACF} provides a natural model for this purpose. However, relaxing the right-angle condition fundamentally changes the boundary behavior of the flow. Under the right-angle condition, the normal direction of the curve at each end point is parallel to the corresponding line. Hence, the prescribed normal velocity automatically keeps the end points on the lines. For general contact angles, by contrast, motion by the prescribed normal velocity alone would generally move the end points away from the lines. Therefore, an appropriate tangential component of the velocity is required at the end points to ensure that they remain on the lines throughout the evolution. This leads us to formulate the following free boundary problem:

  \begin{align}\label{P}\tag{P}
    \begin{cases}
      \langle \partial_t \gamma (u, t), {\bm \nu} \rangle = F_m[k] (u, t) \quad &\text{in} \quad (-1, 1) \times (0, T),\\
      \gamma(-1, t) \in \eta_0, \quad \gamma(1, t) \in \eta_\theta, \quad &\text{in} \quad (0, T),\\
      {\bm \tau}(-1, t) = { \bm v }_{\theta_-}, \quad {\bm \tau}(1, t) = {\bm v}_{\theta + \theta_+}, \quad &\text{in} \quad (0, T),\\
        \partial^{2j-1}_s k (\pm1, t) =0 \quad \text{for} \quad j \in \{ 1, 2, 3, \dots , m \} \quad &\text{in} \quad (0, T) \quad \text{if} \quad m \geq 1,\\
      \gamma(u, 0) = \gamma_0 (u) \quad &\text{in} \quad (-1, 1). 
    \end{cases}
  \end{align}
Unlike the right-angle case, \eqref{P} is no longer the gradient flow of $L[\gamma]$, but rather the $H^{-m}$-gradient flow of 
\[ 
E[\gamma]  := L[\gamma] - \langle \gamma (1), {\bm v}_{\theta + \theta_+} \rangle + \langle \gamma (-1), {\bm v}_{\theta_-} \rangle 
\]
under general contact-angle conditions. The main goal of this paper is to establish global-in-time solvability of \eqref{P} and study the asymptotic behavior of its solutions. 
%
Before stating our main result, we mention the initial conditions. We assume that $\gamma_0 : [-1,1] \to \R^2$ is a smooth regular curve and satisfies
\begin{align}\label{ACF:BC}
    \begin{cases}
        \gamma_0(-1) \in \eta_0, \quad \gamma_0(1) \in \eta_\theta,\\
        {\bm \tau}_0(-1) = { \bm v }_{\theta_-}, \quad {\bm \tau}_0(1) = {\bm v}_{\theta + \theta_+},\\
        \partial^{2j-1}_s k_0 (\pm1) = 0 \quad \text{for}  \quad j \in \{ 1, 2, 3, \dots , m \} \quad \text{if} \quad m \geq 1,  
     \end{cases}
\end{align}   
as compatibility conditions, where $\btau_0$ and $k_0$ denote the unit tangential vector and the curvature of $\gamma_0$. 
For  $m = 0, 1$, the flow \eqref{eq:ACF} for open curves with prescribed contact angles on a straight line was considered by Shimojo--Kagaya~\cite{SK_2016}, Kagaya~\cite{K_2020} and Kagaya--Kohsaka~\cite{KK_2020}, who proved the existence of self-similar solutions that move like traveling waves. 
In \cite{SK_2016} and \cite{K_2020}, the local-in-time solvability of the free boundary problem was proved under the following assumption:
 \begin{align}\label{A1}
     L[\gamma_0] \in (0, \infty), \quad |\partial_u \gamma_0| = \frac{L[\gamma_0]}{2} \quad \text{for all} \quad u \in [-1, 1].
 \end{align} 
In Section \ref{loc}, we prove the local-in-time solvability of the problem \eqref{MP} by essentially the same argument as in \cite{SK_2016} under the assumption \eqref{A1}.
To study the asymptotic behavior of the solutions, we define the following notation$:$
    \begin{align*}
    \mathcal{E}_0 [\gamma] &:= E [\gamma] \int_\gamma (k - \bar{k})^2 \, ds, \quad 
    \mathcal{E}_1 [\gamma] := E [\gamma]^3 \int_\gamma ( \partial_s k)^2 \, ds.
    \end{align*} 
We note that $\mathcal{E}_0 [\gamma]$ and $\mathcal{E}_1[\gamma]$ are scale-invariant quantities. As long as $E[\gamma] > 0$ holds, $\mathcal{E}_0 [\gamma]$ and $\mathcal{E}_1[\gamma]$ measure how close the curvature of $\gamma$ is to a constant. We define $\mathcal{D}$ by
  \begin{align*}
    \mathcal{D} := \left\{ (\sigma, \alpha, \beta) \,\, \middle| \,\,
      \begin{aligned}
       \sigma \in (0, \pi), \,\,\,
       \alpha, \beta \in (0, 2 \pi) \setminus \{ \pi \}, \,\,\,
        | \sigma + \beta - \alpha | < \pi, \,\,\,
        \mu (\sigma, \alpha, \beta) < 1
      \end{aligned}
     \right\},
  \end{align*}
where $\mu (\sigma, \alpha, \beta)$ is defined by
\[
\mu (\sigma, \alpha, \beta) := \sqrt{\frac{\cos^2 \alpha + \cos^2 \beta - 2 \cos \sigma \cos \alpha \cos \beta}{\sin^2 \sigma}}.
\]
For brevity, we write $\mu:=\mu(\theta,\theta_-,\theta_+)$. Here we assume $(\theta, \theta_-, \theta_+) \in \mathcal{D}$. For every admissible curve satisfying the prescribed end point and contact-angle conditions, the condition $\mu < 1$ ensures the positivity of $E[\gamma]$. Indeed, as shown in Lemma~\ref{lem:LleqE}, we obtain
\[
E[\gamma] \geq (1-\mu)L[\gamma].
\]
Thus, the condition $\mu<1$ implies both the positivity of $E[\gamma]$ and the estimate
\[
L[\gamma]\leq\frac{E[\gamma]}{1-\mu} \leq \frac{E[\gamma_0]}{1-\mu}.
\]
Since $E[\gamma]$ is nonincreasing along the flow, this provides a uniform upper bound for $L[\gamma]$. This estimate replaces the monotonicity of $L[\gamma]$, which is generally unavailable because $\eqref{P}$ is not the gradient flow of $L[\gamma]$ when the contact angles are not right angles. Together with the scale-invariant quantities $\mathcal E_0[\gamma]$ and $\mathcal E_1[\gamma]$, it forms the basis of our a priori estimates. The geometric implications of the conditions defining $\mathcal{D}$, particularly for stationary solutions, are described in Corollary~\ref{cor:k>0arc}.
In addition, we impose the following assumption on the initial area $A[\gamma_0]$ and rotation number $\omega [\gamma_0] := \int_{\gamma_0} k_0 \, ds$:
  \begin{align}\label{A2}
   \begin{aligned}
    A[\gamma_0] &\in (0, \infty), \quad \omega[\gamma_0] = \theta + \theta_+ - \theta_-.
   \end{aligned}
   \end{align}
The main result of this paper is stated as follows:
 \begin{theorem}\label{shukekka}
 Let $m \in \N \cup \{ 0 \}$ and $(\theta, \theta_-, \theta_+) \in \mathcal{D}$. Then\, there exists $K^* > 0$ such that the following holds$:$
 Assume $\gamma_0$ is a smooth regular curve satisfying \eqref{ACF:BC}, \eqref{A1}, \eqref{A2} and
   \begin{align}\label{A3}
     \mathcal{E}_0 [\gamma_0] < K^* \quad \text{if} \quad m = 0, \qquad \mathcal{E}_1 [\gamma_0] < K^* \quad \text{if} \quad m \geq 1.
   \end{align}
Then \eqref{P} admits a smooth global-in-time solution $\gamma : (-1, 1) \times [0, \infty) \to \R^2$, which is unique up to time-dependent orientation-preserving reparametrizations fixing the end points. Moreover, there exists a smooth family of orientation-preserving diffeomorphisms $\psi(\cdot, t): [-1,1] \to [-1,1]$ satisfying $\psi(\pm1, t) = \pm1$ and $\psi (u, 0) = u$ such that the reparametrized curve $\gamma(\psi(\cdot,t),t)$ is parametrized with constant speed and converges exponentially in the $C^\infty$-topology, as $t \to \infty$, to the circular arc $\gamma_\infty (\cdot)$ having rotation number $\omega[\gamma_\infty] = \theta + \theta_+ - \theta_-$ and area $A[\gamma_\infty] = A[\gamma_0]$.
  \end{theorem}

Since the geometric evolution problem $\eqref{P}$ prescribes only the normal component of the velocity, uniqueness is understood up to time-dependent orientation-preserving reparametrizations fixing the end points. To analyze $\eqref{P}$, we introduce the modified problem $\eqref{MP}$, in which a suitable tangential velocity is prescribed to preserve the constant-speed parametrization; see Section \ref{sec:set}. The results for $\eqref{P}$ are then obtained from those for $\eqref{MP}$ by a time-dependent reparametrization. 

The rest of this paper is organized as follows. Section 2 introduces the notation and collects auxiliary inequalities and basic evolution identities. Section 3 formulates the suitably modified problem \eqref{MP} and establishes its fundamental properties. Section 4 proves the local well-posedness, global-in-time existence, and convergence of solutions to \eqref{MP}. Finally, Section \ref{sec:pfofshukekka} completes the proof of Theorem~\ref{shukekka}. 


\section{Preliminaries}\label{sec:Pre}

In Section \ref{sec:Pre}, we collect the notation, basic lemmas, and inequalities used throughout this paper.
Let $\gm: [-1, 1] \times [0, T) \to \R^2$ be a family of smooth regular curves. 
The arc length parameter $s$ of $\gm$ is defined by 
\[
s=s(u,t):= \int^u_{-1} |\partial_{\tilde{u}} \gm(\tilde{u},t)| \, d\tilde{u}. 
\]
We denote the integral of the function $f$ defined on $\gm$ with respect to the arc length parameter $s$ of $\gm$ by  
\[
\int_\gm f \, ds := \int^{L[\gm]}_0 f(s) \, ds = \int^1_{-1} f(u) |\partial_u \gm| \, du. 
\]
Moreover, we define the $L^p$-norm with respect to the arc length parameter of $\gm$: 
\begin{align*}
\| f \|_p &:= \Bigl( \int_\gm |f|^p \, ds \Bigr)^\frac{1}{p} \quad &\text{for} \quad 1 \le p < \infty,\\
\| f \|_\infty &:= \esssup_{s \in [0, L[\gamma]]} |f (s)| = \esssup_{u \in [-1, 1]} |f(u)| \quad &\text{for} \quad p = \infty.
\end{align*}
Using it, for all $r, q \in \N$, we define the $W_\gamma^{r, q}$-norm by
\[
\| f \|_{W^{r,q}_\gm} := \sum^{r}_{i=0} \|\partial^i_s f \|_{L^q_\gm},
\] 
where
\[
\| \partial^i_s f \|_{L^q_\gm} := L[\gm]^{i+1-\tfrac{1}{q}}  \| \partial^i_s f \|_q. 
\]
In the following, we collect the inequalities used in this paper:
\begin{lemma}[\cite{WW_2024}]\label{lem:Poincare1}
Let $f : (0, L) \to \R$ satisfy $f \in H^1(0,L)$ and $\int_0^L f \, ds = 0$. Then
\begin{align*}
\int_0^L f^2 \, ds \le \frac{L^2}{\pi^2} \int_0^L f_s^2 \, ds, \quad
\| f \|_\infty^2 \le \frac{2 L}{\pi} \int_0^L f_s^2 \, ds.
\end{align*}
\end{lemma}

\begin{lemma}[\cite{WW_2024}]\label{lem:Poincare2}
Let $f : (0, L) \to \R$ satisfy $f \in H^1(0,L)$ and $f(0) = f(L) = 0$. Then 
\begin{align*}
\int_0^L f^2 \, ds \le \frac{L^2}{\pi^2} \int_0^L f_s^2 \, ds, \quad
\| f \|_\infty^2 \le \frac{L}{\pi} \int_0^L f_s^2 \, ds.
\end{align*}
\end{lemma}

We use the following interpolation inequality (see, e.g., \cite[Appendix C]{DP_2014}). 
\begin{proposition} \label{prop:interpolation}
Let $\gm : [-1, 1] \to \R^2$ be a smooth regular curve. 
Then for all $m \in \N$, $p \ge 2$, $0 \le j < m$, and any $f$ with $\| f \|_{W^{m,2}_\gm} < \infty$, we have 
\[
\| \partial^j_s f \|_{L^p_\gm} \le C \| f \|_{L^2_\gm}^{1-\alpha} \| f \|_{W^{m,2}_\gm}^\alpha
\]
with $C=C(j, m, p) > 0$ and 
\[
\alpha = \dfrac{1}{m} \Bigl( j + \dfrac{1}{2} - \dfrac{1}{p} \Bigr).  
\]
\end{proposition}

Applying Proposition~\ref{prop:interpolation}, we obtain the following interpolation inequality for $f$ satisfying $\int_\gm f \, ds = 0$ and $\partial_s^{2l -1} f (0)$ = $\partial_s^{2l -1} f (L[\gamma]) = 0$ for every $l \in \N$ satisfying $l \leq \frac{m + 1}{2}$.
\begin{lemma} \label{lem:k-interpolation}
Let $\gm : [-1, 1] \to \R^2$ be a smooth regular curve. 
Then for all $m \in \N$, $p \ge 2$, $0 \le j < m$, and any $f$ satisfying $\| f \|_{W^{m,2}_\gm} < \infty$ and 
\begin{align}\label{eq:k-interpolation}
\int_\gamma f \, ds &= 0 \quad \text{and} \quad
\partial_s^{2l - 1} f (0) = \partial_s^{2l - 1} f (L[\gamma]) = 0, 
\end{align}
for every $l \in \N$ satisfying $l \leq \frac{m + 1}{2}$, we have 
\[
\| \partial^j_s f \|_p \le M \| f \|_2^{1-\alpha} \| \partial^m_s f \|_2^\alpha
\]
with $M=M(j, m, p)>0$ and 
\[
\alpha = \dfrac{1}{m} \Bigl( j + \dfrac{1}{2} - \dfrac{1}{p} \Bigr).  
\]
\end{lemma}

\begin{proof}
Applying Proposition~\ref{prop:interpolation} to $f$, we find $C(j,m,p)>0$ such that  
\[
\| \partial^j_s f \|_{L^p_\gm} \le C \| f \|_{L^2_\gm}^{1-\alpha} \| f \|_{W^{m,2}_\gm}^\alpha.
\]
Thus we have
\[
\| \partial^j_s f \|_p 
 \le C L[\gm]^{-j+\frac{1}{p}-\frac{1}{2}} \| f \|_2^{1-\alpha} \sum^m_{i=0} L[\gm]^{i \alpha} \| \partial^i_s f \|_2^\alpha.
\]
Since $f$ satisfies \eqref{eq:k-interpolation}, it follows from Lemmas \ref{lem:Poincare1} and \ref{lem:Poincare2} that 
\[
\| \partial^i_s f \|_2 \le \Bigl[ \dfrac{L[\gm]}{\pi} \Bigr]^{m-i} \| \partial^m_s f \|_2
\]
for each $i \in \{0, 1, \ldots, m-1\}$. Thus we obtain   
\begin{equation*}
\begin{aligned}
\| \partial^j_s f \|_p 
 &\le C (m+1) L[\gm]^{-j+\frac{1}{p}-\frac{1}{2}+m \alpha} \| f \|_2^{1-\alpha} \| \partial^m_s f \|_2^\alpha \\
 &= C(m+1) \| f \|_2^{1-\alpha} \| \partial^m_s f \|_2^\alpha. 
\end{aligned}
\end{equation*}
\end{proof}


In what follows, we collect the fundamental properties of the smooth solution to the following general geometric evolution equation for open curves:
  \begin{align}\tag{GE}\label{GF}
    \partial_t \gamma = F \bnu + G \btau \quad \text{in} \quad (-1, 1) \times (0, T),
  \end{align}
where $F : [-1, 1] \times (0, T) \to \R$ and $G : [-1, 1] \times (0, T) \to \R$ denote the normal velocity of $\gamma$ and the tangential velocity of $\gamma$, respectively. We assume that the solution to \eqref{GF} satisfies the following$\colon$
  \begin{align}\label{BCofGF}
    \begin{cases}
      \gamma(-1, t) \in \eta_0, \quad \gamma (1, t) \in \eta_\theta, \quad &\text{in} \quad (0, T),\\
       \btau (-1, t) = {\bm v}_{\theta_-}, \quad  \btau (1, t) = {\bm v}_{\theta + \theta_+}, \quad &\text{in} \quad (0, T),\\
       \gamma (u, 0) = \gamma_0 (u) \quad &\text{in} \quad (-1, 1).
    \end{cases}
   \end{align}
We assume that $\gamma_0 : (-1, 1) \to \R^2$ satisfies
  \begin{align*}
   \begin{cases}
    \gamma_0(-1) \in \eta_0, \quad \gamma_0 (1) \in \eta_\theta,\\
       \btau_0 (-1) = {\bm v}_{\theta_-}, \quad  \btau_0 (1) = {\bm v}_{\theta + \theta_+},\\
     \end{cases}
  \end{align*}
as compatibility conditions. Here, we prove the following commutator rule for the solution to \eqref{GF}$\colon$
\begin{lemma}\label{lem:commutator}
Let $\gamma : (-1, 1) \times [0, T) \to \R^2$ be a smooth solution to \eqref{GF} satisfying \eqref{BCofGF}. Then we have$\colon$ 
  \begin{align}
    \partial_t ds &= (- k F + \partial_s G) ds, \label{eq:commutator1}\\
    \partial_t \partial_s - \partial_s \partial_t &= ( k F - \partial_s G) \partial_s. \label{eq:commutator2}
      \end{align}
\end{lemma}
\begin{proof}
From the Frenet--Serret formula \eqref{FS} and \eqref{GF}, we have
  \begin{align}\label{eq:commutator3}
   \begin{aligned}
    \partial_t |\partial_u \gamma| 
      = \frac{\la \partial_u \gamma, \partial_u \partial_t \gamma \ra}{|\partial_u \gamma|}
      = \la \btau, \partial_s \partial_t \gamma \ra
      = (-k F + \partial_s G) |\partial_u \gamma|.
     \end{aligned}
  \end{align}
Together with $ds = |\partial_u \gamma| du$, this yields \eqref{eq:commutator1}. We observe from $\partial_s = \frac{\partial_u}{|\partial_u \gamma|}$ that
  \begin{align*}
    \partial_t \partial_s 
      = \partial_t \left(\frac{1}{|\partial_u \gamma|} \right) \partial_u + \frac{1}{|\partial_u \gamma|} \partial_u \partial_t
      = - \frac{\partial_t |\partial_u \gamma|}{|\partial_u \gamma|} \partial_s + \partial_s \partial_t.
  \end{align*}
This together with \eqref{eq:commutator3} implies that
  \begin{align*}
    \partial_t \partial_s - \partial_s \partial_t = (k F - \partial_s G) \partial_s.
  \end{align*}
This completes the proof.
\end{proof}
From Lemma \ref{lem:commutator}, we derive the differential equations of the curvature $k$, the unit normal vector $\bnu$, and the unit tangential vector $\btau$.
\begin{lemma}\label{lem:ptk}
Let $\gamma : (-1, 1) \times [0, T) \to \R^2$ be a smooth solution to \eqref{GF} satisfying \eqref{BCofGF}. Then we have$\colon$ 
  \begin{align}
    \partial_t \btau &= (\partial_s F + k G) \bnu, \label{eq:pttau} \\
    \partial_t \bnu &= - (\partial_s F + k G) \btau, \label{eq:ptnu} \\
    \partial_t k &= \partial_s^2 F + k^2 F + \partial_s k G. \label{eq:ptk}
  \end{align}
\end{lemma}
\begin{proof}
From Lemma \ref{lem:commutator}, we have
  \begin{align*}
    \partial_t \btau 
      &= \partial_s \partial_t \gamma + (k F - \partial_s G) \partial_s \gamma\\
      &= \partial_s F \bnu - k F \btau + \partial_s G \btau + k G \bnu + k F \btau - \partial_s G \btau\\
      &= (\partial_s F  + k G) \bnu.
  \end{align*}
Thus we obtain \eqref{eq:pttau}. Moreover, we deduce from \eqref{eq:pttau} and $\la \btau, \bnu \ra = 0$ that
  \begin{align*}
    \la \btau, \partial_t \bnu \ra = - \la \partial_t \btau, \bnu \ra = - (\partial_s F + k G).
  \end{align*}
Since $|\bnu| = 1$ implies $\la \bnu, \partial_t \bnu \ra = 0$, we have
  \begin{align*}
    \partial_t \bnu = \la \bnu, \partial_t \bnu \ra \bnu + \la \btau, \partial_t \bnu \ra \btau = - (\partial_s F + k G) \btau, 
  \end{align*}
and then \eqref{eq:ptnu} follows. Finally, we prove \eqref{eq:ptk}. For the curvature vector $\bkappa = \partial_s^2 \gamma$, we observe from \eqref{FS}, \eqref{eq:pttau} and Lemma \ref{lem:commutator} that
  \begin{align}\label{eq:ptkappa}
    \begin{aligned}
    \partial_t \bkappa 
      &= \partial_s \partial_t \btau + (k F - \partial_s G) \partial_s \btau\\
      &= \partial_s \left( (\partial_s F + k G) \bnu \right) - (k^2 F - k \partial_s G) \bnu\\
      &= (\partial_s^2 F + k^2 F + \partial_s k G) \bnu - (k \partial_s F + k^2 G) \btau.
     \end{aligned}
    \end{align}
Since $k = \la \bkappa, \bnu \ra$, it follows from \eqref{eq:ptnu} and \eqref{eq:ptkappa} that
  \begin{align*}
    \partial_t k 
      = \la \partial_t \bkappa, \bnu \ra + \la \bkappa, \partial_t \bnu \ra
      = \partial_s^2 F + k^2 F + \partial_s k G.
  \end{align*}
This completes the proof.
\end{proof}
From \eqref{BCofGF}, we prove that the rotation number $\omega[\gamma]$ is invariant$\colon$
\begin{lemma}\label{lem:rot-invariant}
Let $\gamma : (-1, 1) \times [0, T) \to \R^2$ be a smooth solution to \eqref{GF} satisfying \eqref{BCofGF}. Then 
  \begin{align*}
    \omega[\gamma](t) = \omega[\gamma_0] \quad \text{for all} \quad t \in [0, T). 
  \end{align*}
\end{lemma}
\begin{proof}
From \eqref{BCofGF}, we see that
 \begin{align*}
   0 = \frac{d}{dt} \la \btau(-1, t), {\bm v}_{\theta_- + \frac{\pi}{2}} \ra
      = \la \partial_t \btau (-1, t), \bnu(-1, t) \ra
 \end{align*} 
Similarly, we obtain $\la \partial_t \btau (1, t), \bnu (1, t) \ra$ = 0. This together with Lemma \ref{lem:ptk} implies that 
  \begin{align*}
    \frac{d}{dt} \int_\gamma k \, ds
      &= \int_\gamma ( \partial_t k - k^2 F + k \partial_s G ) \, ds\\
      &= \int_\gamma ( \partial_s^2 F + \partial_s k G + k \partial_s G) \, ds\\
      &= \int_\gamma \partial_s (\partial_s F + k G) \, ds\\
      &= \int_\gamma \partial_s \la \partial_t \btau, \bnu \ra \, ds
      = \left[ \la \partial_t \btau, \bnu \ra \right]_{-1}^1
      = 0.
  \end{align*}
This completes the proof.
\end{proof}

At the end of this section, we derive the lower bound for the length of $\gamma$ satisfying \eqref{BCofGF}.
We make use of the isoperimetric inequality for open curves proved by Miura--Okabe~\cite[Theorem 3.1]{MO}. 
Here we define 
\[
X_\theta := \{ \gamma \in W^{1, \infty}((0, 1) ; \R^2) \mid \gamma(0) \in \eta_0, \,\, \gamma(1) \in \eta_\theta, \,\, |\gamma(0)|=|\gamma(1)|, \,\, L[\gamma]>0 \} 
\] 
for $\theta \in [0, 2\pi]$. 

\begin{proposition}[{\cite[Theorem 3.1]{MO}}] \label{prop:isoperimetric}
For any $\theta \in (0, 2\pi]$ and $\gamma \in X_\theta$, the following isoperimetric inequality holds$\colon$ 
\[
L[\gm]^2 \ge 2 \theta A[\gm]. 
\]
The equality is attained if and only if $\gm$ is a counterclockwise circular arc of central angle $\theta;$  
in particular, if $\theta \in (0, 2 \pi)$, then the arc needs to be centered at the origin.  
\end{proposition}

Modifying the set $X_\theta$ in Proposition~\ref{prop:isoperimetric}, we obtain an isoperimetric inequality for $\gamma$ satisfying \eqref{BCofGF}: 
\begin{lemma}\label{lem:Lbdd}
Let $\theta \in (0, \pi]$ and $\theta_-, \theta_+ \in (0, 2 \pi) \setminus \{ \pi \}$. Let
  \begin{align*}
    X_{\theta, \theta_-, \theta_+} := \left\{ \gamma \in C^{1}([0, 1]; \R^2) \,\, \middle| \,\,
    \begin{aligned}
     &\gamma(0) \in \eta_0, \,\,  \gamma (1) \in \eta_\theta, \,\, L[\gamma] > 0,\\
     &{\bm \tau}(0) = { \bm v }_{\theta_-}, \,\, {\bm \tau}(1) = {\bm v}_{\theta + \theta_+}
     \end{aligned}
     \right\}.
  \end{align*}
For any $\gamma \in X_{\theta, \theta_-, \theta_+}$,
  \begin{align}\label{eq:Lbdd}
    L[\gamma]^2 \geq 2 \theta A[\gamma].
  \end{align}
\end{lemma}
\begin{proof}
The proof is based on reflecting the curve with respect to the skew lines and applying Proposition~\ref{prop:isoperimetric}. The right-angle case was proved by Hiroi--Okabe~\cite[Lemma 4.2]{HO_2025}. We therefore sketch the proof of Lemma \ref{lem:Lbdd} here.
Set 
\[
\eta_\alpha^+ := \{ \lambda {\bm v}_\alpha \mid \lambda > 0 \}, \qquad
\eta_\alpha^- := \{ \lambda {\bm v}_\alpha \mid \lambda < 0 \},  
\]
and we split the proof into the following four cases:

\noindent
{\bf Case 1. }{\sl We consider the case of $\gamma(1) \in \eta_\theta^+$.} 

\noindent
{\bf Case 2. }{\sl We consider the case of $\gm(1) = (0,0)$ and $\gm(0) \in \eta_0^+ \cup \{(0,0) \}$.}

\noindent
{\bf Case 3. }{\sl We consider the case of $\gm(1) = (0,0)$ and $\gm(0) \in \eta_0^-$.}

\noindent
{\bf Case 4. }{\sl We consider the case of $\gm(1) \in \eta_\theta^-$.} 

\noindent
For each case, we extend the curves by reflecting with respect to the skew lines. Here, we prove \eqref{eq:Lbdd} in Case 1 (see Fig. \ref{fig2}). 
\begin{figure}[htbp]
  \centering
 \includegraphics[width=0.42\textwidth]{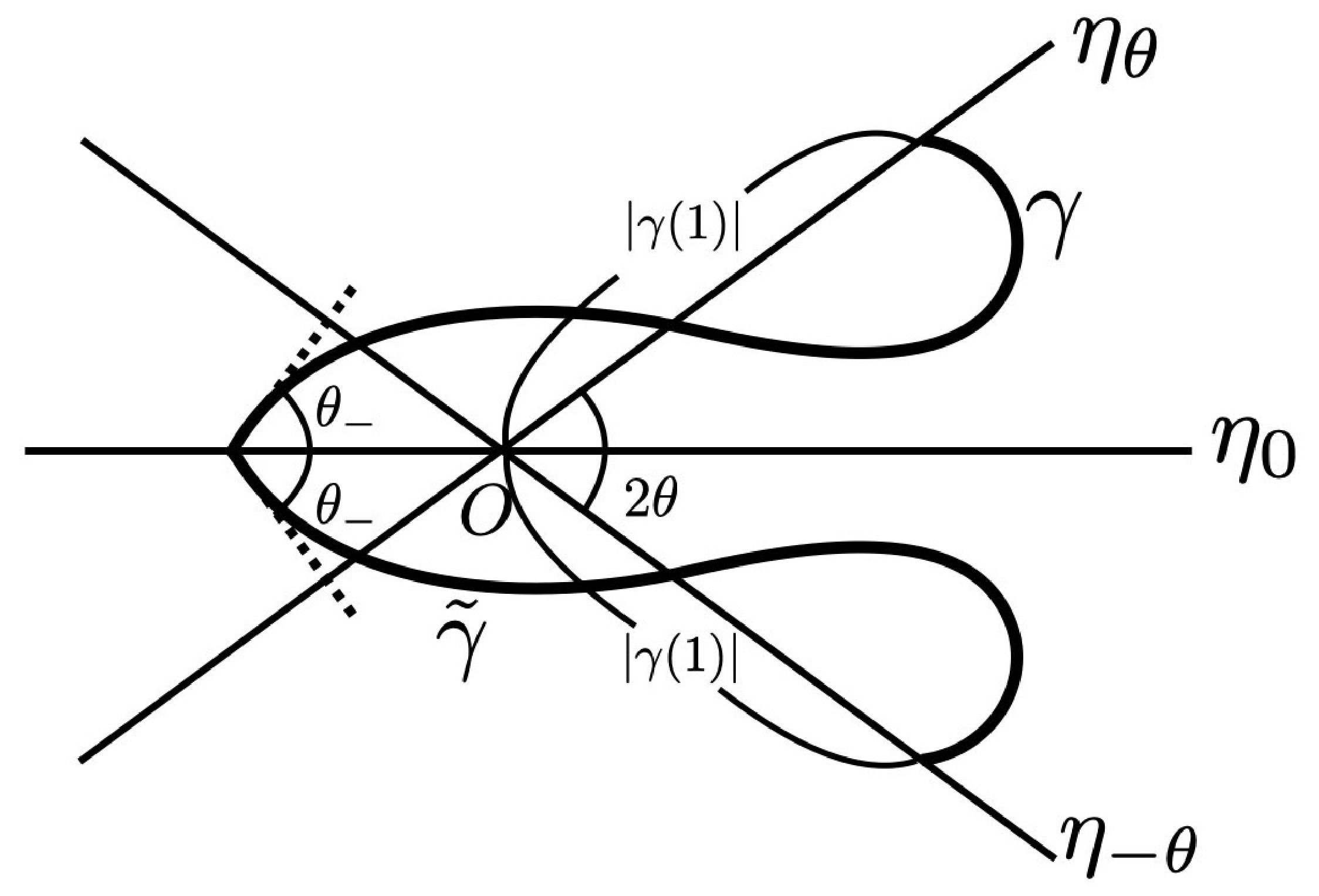}
  \caption{$\tilde{\gamma}$ in Case 1}
  \label{fig2}
\end{figure}

We define $\tilde{\gamma}$ by 
\begin{equation} \label{eq:def-t-gm}
\begin{aligned}
\tilde{\gamma}(x) = 
\begin{cases}
(\gamma_1(-x), -\gamma_2(-x)) \ \quad &\text{for} \quad x \in [-1, 0], \\
(\gamma_1(x), \gamma_2(x)) \quad &\text{for} \quad x \in [0, 1].
\end{cases}
\end{aligned}
\end{equation}
It is clear that 
\[
\tilde{\gm}(-1) \in \eta_{-\theta}^+, \quad \tilde{\gm}(1) \in \eta_\theta^+, \quad |\tilde{\gm}(-1)|=|\tilde{\gm}(1)|, \quad \text{and} \quad L[\tilde{\gm}] = 2 L[\gm]>0. 
\]
Once we verify that $\tilde{\gamma} \in W^{1, \infty} ((-1, 1); \R^2)$, Proposition~\ref{prop:isoperimetric} yields
  \begin{align}\label{eq:pre-case-1}
    L[\tilde{\gamma}]^2 \geq 4 \theta A[\tilde{\gamma}].
  \end{align}
Indeed, by the definition of $\tilde{\gamma}$, we have
  \begin{equation*}
    \begin{aligned}
      A[\tilde{\gamma}] 
       &= -\frac{1}{2} \int_{-1}^1 \Bigl\langle \tilde{\gamma}(x), R \frac{\tilde{\gamma}'(x)}{|\tilde{\gamma}'(x)|} \Bigr\rangle |\tilde{\gamma}'(x)| \, dx\\
       &= -\frac{1}{2} \int_{-1}^0 (\gamma_1(-x), -\gamma_2(-x)) \cdot \dfrac{R (-\gamma_1'(-x), \gamma_2'(-x))}{|\gamma'(-x)|} |\gamma'(-x)| \, dx + A[\gm]\\
       &= -\frac{1}{2} \int_0^1 \frac{1}{|\gamma'(y)|} (-\gamma_1(y) \gamma_2'(y) + \gamma_2(y) \gamma_1'(y) ) |\gamma'(y)| \, dy + A[\gamma]\\
       &= -\frac{1}{2} \int_{0}^1 \Bigl\langle \gamma(y), R \frac{\gamma'(y)}{|\gamma'(y)|} \Bigr\rangle |\gamma'(y)| \, dy + A[\gm] 
         = 2A[\gamma].
    \end{aligned}
  \end{equation*}
Since $L[\tilde{\gm}]=2L[\gm]$, this together with \eqref{eq:pre-case-1} implies that 
\[
L[\gm]^2 \ge 2 \theta A[\gm].
\]
Now, we prove $\tilde{\gamma} \in W^{1, \infty}((-1, 1); \R^2)$. From $\gamma \in C^1 ([0, 1]; \R^2)$ and \eqref{eq:def-t-gm}, $\tilde{\gamma}$ is locally Lipschitz continuous in $(-1, 1) \setminus \{ 0 \}$. Thus it suffices to check the local Lipschitz continuity of $\tilde{\gamma}$ at $x = 0$. We define $R_\alpha$ by
  \begin{align*}
    R_\alpha : = 
      \begin{pmatrix}
      \cos \alpha &-\sin \alpha \\
      \sin \alpha &\cos \alpha
      \end{pmatrix},      
  \end{align*}
and consider $R_{\frac{\pi}{2}} \tilde{\gamma}$. We prove the local Lipschitz continuity of $R_{\frac{\pi}{2}} \tilde{\gamma}$ at $x = 0$. We divide the argument into two steps.

\noindent
{\bf Step 1.} {\sl We prove $R_{\frac{\pi}{2}} \tilde{\gamma}$ can be represented as a graph in a neighborhood of $x = 0$.}
From \eqref{BCofGF}, we have
  \begin{align}\label{eq:tanof0}
    \lim_{x \to +0} \gamma' (x) = (\cos \theta_-, \sin \theta_-).
  \end{align}
By \eqref{eq:def-t-gm}, we express $R_{\frac{\pi}{2}} \tilde{\gamma}$ as
  \begin{align*}
    R_{\frac{\pi}{2}} \tilde{\gamma} := (\zeta_1 (x), \zeta_2 (x)) =
      \begin{cases}
        (\gamma_2 (-x), \gamma_1 (-x)) \quad \text{for} \quad x \in [-1, 0],\\
        (- \gamma_2 (x), \gamma_1 (x)) \quad \text{for} \quad x \in [0, 1].
      \end{cases}
  \end{align*}  
 Here we note $\zeta_1 (0) = 0$. From \eqref{eq:tanof0}, we have
  \begin{align*}
    \lim_{x \to +0} \frac{d}{dx}(- \gamma_2 (x)) = \lim_{x \to -0} \frac{d}{dx}(\gamma_2 (-x)) = - \sin \theta_-.
  \end{align*}
This together with $\theta_- \in (0, 2\pi) \setminus \{ \pi \}$ implies that
  \begin{align*}
    \frac{d}{dx} \zeta_1 (0) = - \sin \theta_- \neq 0.
  \end{align*}
Thus there exists some $\delta_0 > 0$ and $f : [-\delta_0, \delta_0] \to \R$ such that
  \begin{align*}
    R_{\frac{\pi}{2}} \tilde{\gamma}(y) = (y, f(y)) \quad \text{for} \quad y \in [-\delta_0, \delta_0].
  \end{align*}
   
\noindent
{\bf Step 2.} {\sl We prove $f$ is locally Lipschitz continuous at $y = 0$.}
From $\gamma \in C^1([0, 1]; \R^2)$, $f$ is locally Lipschitz continuous on $(-\delta, \delta) \setminus \{ 0\}$. From \eqref{BCofGF}, we have
  \begin{align}\label{eq:tanoff}
    \lim_{y \to +0} \frac{d}{dy} f (y) = \cot \theta_-, \quad 
    \lim_{y \to -0} \frac{d}{dy} f (y) &= - \cot \theta_-.
  \end{align}
Fix $\varepsilon > 0$ arbitrarily. From \eqref{eq:tanoff}, there exists $\delta \in (0, \delta_0]$ such that
  \begin{align*}
    \left| \frac{f(y_+) - f(0)}{y_+} - \cot \theta_- \right| &< \varepsilon \quad &\text{for all} \quad y_+ \in [0, \delta),\\
    \left| \frac{f(y_-) - f(0)}{y_-} + \cot \theta_- \right| &< \varepsilon \quad &\text{for all} \quad y_- \in (-\delta, 0].\\
  \end{align*}
Thus we have
  \begin{align*}
     | f(y_+) - f(0) | 
       &< \left( \left| \cot \theta_- \right| + \varepsilon \right) |y_+|\quad &\text{for all} \quad y_+ \in [0, \delta),\\
     | f(y_-) - f(0) | 
       &<\left( \left| \cot \theta_- \right| + \varepsilon \right) |y_-| \quad &\text{for all} \quad y_- \in (-\delta, 0].\\
  \end{align*}
This implies that
  \begin{align*}
    |f(y_+) - f(y_-)| 
      &\leq |f(y_+)-f(0)|+|f(y_-)-f(0)|\\
      &< 2 \left( \left| \cot \theta_- \right| + \varepsilon \right) \left( |y_+| + |y_-| \right)\\
      &= 2 \left( \left| \cot \theta_- \right| + \varepsilon \right) |y_+ - y_-|
   \end{align*} 
for all $y_+ \in [0, \delta)$ and $y_- \in (-\delta, 0]$. Thus $f$ is locally Lipschitz continuous at $y = 0$.

From Steps 1 and 2, we see that $R_{\frac{\pi}{2}} \tilde{\gamma}$ is locally Lipschitz continuous at $x = 0$ and this implies $\tilde{\gamma} \in W^{1, \infty} ((-1, 1); \R^2)$. Cases 2-4 can be proved by a similar argument (see \cite[Lemma 4.1]{HO_2025}).
\end{proof}


\section{Formulation of the modified problem}\label{sec:set}

To construct solutions to \eqref{P}, we consider the following initial-boundary value problem:
\begin{align}\label{MP}\tag{MP}
\begin{cases}
\partial_t \gamma (u, t) = F_m[k] {\bm \nu} + G_m[k] {\bm \tau} \quad &\text{in} \quad (-1, 1) \times (0, T),\\
\gamma(-1, t) \in \eta_0, \quad \gamma(1, t) \in \eta_\theta, \quad &\text{in} \quad (0, T),\\
{\bm \tau} (-1, t) = { \bm v }_{\theta_-}, \quad {\bm \tau} (1, t) = {\bm v}_{\theta + \theta_+}, \quad &\text{in} \quad (0, T),\\
\partial^{2j-1}_s k (\pm1, t) =0 \quad ( j \in \{ 1, 2, 3, \dots , m \} ) \quad &\text{in} \quad (0, T) \quad \text{if} \quad m \geq 1,\\
\gamma(u, 0) = \gamma_0 (u) \quad &\text{in} \quad (-1, 1),\\
\end{cases}
\end{align}
where $G_m[k]: [-1, 1] \times [0, T) \to \R$ is given by
\begin{equation}\label{def:Gmk}
\begin{aligned}
G_m[k]
&:= \int_0^s k F_m[k] \, d\tilde{s}\\
&+ \frac{s}{L[\gamma]} \left(- \int_\gamma k F_m[k]\, ds - F_m[k] (1, t) \cot \theta_+ + F_m[k](-1, t) \cot \theta_-  \right)\\
&-F_m[k](-1, t) \cot \theta_-.
\end{aligned}
\end{equation}
If \eqref{MP} admits a solution, then it also satisfies \eqref{P}. Therefore, it suffices to study \eqref{MP} to prove Theorem~\ref{shukekka}.
We derive the modified system \eqref{MP} in the next subsection.


\subsection{Derivation of \eqref{MP}}\label{sec:for}

This section provides the preliminary calculations needed to study \eqref{MP}. We assume that $\gamma_0 : (-1,1) \to \R^2$ is a smooth immersed curve and satisfies \eqref{ACF:BC}. 
We derive $G_m[k]$ from the boundary condition of \eqref{MP} so that the solution to \eqref{MP} satisfies the following:
\begin{align}\label{eq:u1}
|\partial_u \gamma| = \frac{L[\gamma]}{2} \quad \text{for all} \quad (u, t) \in [-1, 1] \times [0, T).
\end{align}
Assume $\gamma = (\gamma_1, \gamma_2)$ satisfies the boundary conditions of \eqref{MP}, \eqref{eq:u1} and
\begin{align}\label{eq:alpha1}
\partial_t \gamma = F_m[k] {\bm \nu} + \alpha {\bm \tau} \quad \text{in} \quad [-1, 1] \times [0,T),
\end{align}
where $\alpha: [-1, 1] \times [0, T) \to \R$ denotes the tangential velocity of $\gamma$. 
From \eqref{eq:u1} and \eqref{eq:alpha1}, we have
\begin{align*}
\frac{L[\gamma] \frac{d}{dt} L[\gamma]}{2} 
&= \partial_t |\partial_u \gamma|^2
= 2\langle \partial_u \partial_t \gamma, \partial_u \gamma \rangle \\
&=\frac{L[\gamma]^2}{2} \langle \partial_s ( F_m[k] {\bm \nu} + \alpha {\bm \tau}), {\bm \tau} \rangle\\
&=\frac{L[\gamma]^2}{2} (-k F_m[k] + \partial_s \alpha).
\end{align*}
Thus, we have
\begin{align}\label{eq:alpha2}
\partial_u \alpha = \frac{L[\gamma] k F_m[k] + \frac{d}{dt} L[\gamma]}{2}.
\end{align}
Next, we derive $\frac{d}{dt} L[\gamma]$ using the boundary conditions of \eqref{MP}. From \eqref{eq:alpha1},
\begin{align*}
{\bm \nu} (-1, t) &= (- \sin \theta_-, \cos \theta_-) \quad \text{and} \quad
{\bm \tau} (-1, t) = (\cos \theta_-, \sin \theta_-),
\end{align*}
we have
\begin{align*}
{}^t\!\left( \frac{d}{dt} \gamma (-1, t) \right)
=
\begin{pmatrix}
- F_m[k] (-1, t) \sin \theta_- + \alpha(-1, t) \cos \theta_- \\
F_m[k] (-1, t) \cos \theta_- + \alpha (-1, t) \sin \theta_-
\end{pmatrix}.
\end{align*}
This together with $\gamma (-1, t) \in \eta_0$ implies that
\begin{equation}\label{eq:alpha3}
\begin{aligned}
&F_m[k] (-1, t) \cos \theta_- + \alpha (-1, t) \sin \theta_- = 0.
\end{aligned}
\end{equation}
On the other hand, from \eqref{eq:alpha1},
\begin{align*}
{\bm \nu} (1, t) = (- \sin (\theta + \theta_+), \cos (\theta + \theta_+))
\quad \text{and} \quad
{\bm \tau} (1, t) = (\cos (\theta + \theta_+), \sin (\theta + \theta_+)),
\end{align*}
we have
\begin{align}\label{eq:alpha4}
{}^t\!\left( \frac{d}{dt} \gamma (1, t) \right)
=
\begin{pmatrix}
- F_m[k] (1, t) \sin (\theta + \theta_+) + \alpha(1, t) \cos (\theta + \theta_+) \\
F_m[k] (1, t) \cos (\theta + \theta_+) + \alpha (1, t) \sin (\theta + \theta_+)
\end{pmatrix}.
\end{align}
Here, we define $R_\alpha$ by
\begin{align*}
R_\alpha =
\begin{pmatrix}
\cos \alpha & -\sin \alpha\\
\sin \alpha & \cos \alpha
\end{pmatrix}
\end{align*}
and consider
\begin{align*}
R_{-\theta} {}^t\!\left( \frac{d}{dt} \gamma (1, t) \right)
&=
\begin{pmatrix}
\cos \theta & \sin \theta\\
-\sin \theta & \cos \theta
\end{pmatrix}
\begin{pmatrix}
\frac{d}{dt} \gamma_1 (1, t)\\
\frac{d}{dt} \gamma_2 (1, t)
\end{pmatrix}\\
&=
\begin{pmatrix}
\cos \theta \frac{d}{dt} \gamma_1(1, t)
+\sin \theta \frac{d}{dt} \gamma_2 (1, t)\\
-\sin \theta \frac{d}{dt} \gamma_1(1, t)
+\cos \theta \frac{d}{dt} \gamma_2 (1, t)
\end{pmatrix}.
\end{align*}
From $\gamma(1, t) \in \eta_\theta$ and \eqref{eq:alpha4}, we have
\begin{align}\label{eq:alpha5}
\begin{aligned}
0
&= -\sin \theta \frac{d}{dt} \gamma_1(1, t)
+\cos \theta \frac{d}{dt} \gamma_2 (1, t)\\
&= F_m[k] (1, t) \sin (\theta + \theta_+) \sin \theta
-\alpha (1, t) \cos (\theta + \theta_+) \sin \theta\\
&\quad
+F_m[k] (1, t) \cos (\theta + \theta_+) \cos \theta
+\alpha (1, t) \sin (\theta + \theta_+) \cos \theta\\
&= F_m[k] (1, t) \cos \theta_+
+\alpha (1, t) \sin \theta_+.
\end{aligned}
\end{align}
Thus, from \eqref{eq:alpha3} and \eqref{eq:alpha5}, we have
\begin{align}\label{eq:alphapm1}
\alpha(\pm1, t) = - F_m[k] (\pm1, t) \cot \theta_{\pm}.
\end{align}
Thus, from Lemma \ref{lem:commutator}, we have
\begin{equation}
\begin{aligned}\label{eq:dtLGamma}
\frac{d}{dt} L[\gamma]
&= \int_\gamma \langle \partial_s \partial_t \gamma, \partial_s \gamma \rangle \, ds\\
&= - \int_\gamma k F_m[k] \, ds
- F_m[k](1, t) \cot \theta_+
+ F_m[k](-1, t) \cot \theta_-.
\end{aligned}
\end{equation}
Integrating both sides of \eqref{eq:alpha2} and substituting \eqref{eq:dtLGamma}, we derive $\alpha = G_m[k]$.

\begin{remark}\label{rem:pm1Gamma}
Let $\theta \in (0, \pi)$ and let $\theta_\pm \in (0, 2\pi) \setminus \{ \pi \}$. Then the end points of a solution to \eqref{MP} evolve according to the following equations:
\begin{align}\label{eq:pm1Gamma}
\partial_t \gamma (-1, t) = - \frac{F_m[k](-1, t)}{\sin \theta_-} {\bm v}_0, \quad \partial_t \gamma (1, t) = - \frac{F_m[k](1, t)}{\sin \theta_+} {\bm v}_\theta,
\end{align}
for all $t \in [0, T)$. Indeed, from the boundary conditions in \eqref{MP}, we have
\begin{align*}
&{\bm \tau} (-1, t) = (\cos \theta_-, \sin \theta_-), \quad {\bm \nu} (-1, t) = (- \sin \theta_-, \cos \theta_-),\\
&{\bm \tau} (1, t) = (\cos(\theta + \theta_+), \sin (\theta + \theta_+)), \quad {\bm \nu} (1, t) = (- \sin (\theta + \theta_+), \cos (\theta + \theta_+)),
\end{align*}
for all $t \in [0, T)$. This together with \eqref{eq:alphapm1} implies that
\begin{align*}
\partial_t \gamma (-1, t) 
  &= F_m[k] (-1, t) {\bm \nu} (-1, t) - F_m[k] (-1, t) \cot \theta_- {\bm \tau} (-1, t)\\
  &= - \frac{F_m[k] (-1, t)}{\sin \theta_-} (1, 0) = - \frac{F_m[k] (-1, t)}{\sin \theta_-} {\bm v}_0,
\end{align*}
and
\begin{align*}
\partial_t \gamma (1, t) 
  &= F_m[k] (1, t) {\bm \nu} (1, t) - F_m[k] (1, t) \cot \theta_+ {\bm \tau} (1, t)\\
  &= - \frac{F_m[k] (1, t)}{\sin \theta_+} \left( \cos \theta, \sin \theta \right) = - \frac{F_m[k] (1, t)}{\sin \theta_+} {\bm v}_\theta,
\end{align*}
for all $t \in [0, T)$. This means that the end points of a solution to \eqref{MP} move along the skew lines.
\end{remark}


\subsection{Fundamental properties of \eqref{MP}}\label{sec:pre}
In what follows, we collect the fundamental properties of solutions of \eqref{MP}.
We prepare the commutator rule for the solution to \eqref{MP}. For $(u, t) \in (-1, 1) \times [0, T)$, see Lemma~\ref{lem:commutator}. For $(u, t) \in \{-1, 1\} \times [0, T)$, from \eqref{eq:pm1Gamma}, we obtain the following commutator rule:
\begin{lemma}\label{lem:pm1commutator}
Let $m \in \N \cup \{ 0 \}$, $\theta \in (0, \pi]$ and $\theta_-, \theta_+ \in (0, 2\pi) \setminus \{ \pi \}$. Assume $\gamma : (-1,1) \times [0, T) \to \R^2$ is a solution to \eqref{MP}. Then the following holds$:$
\begin{align*}
\partial_t \partial_s = \partial_s \partial_t
\quad
\text{for all}
\quad
(u, t) \in \{-1, 1\} \times [0, T).
\end{align*}
\end{lemma}

\begin{proof}
For a function $f$, we observe from
$\partial_s = \frac{\partial_u}{|\partial_u \gamma|}$
and \eqref{eq:pm1Gamma} that
\begin{align*}
\partial_t \partial_s f (-1, t)
&=
\partial_t
\left(
\frac{\partial_u f (-1, t)}
{|\partial_u \gamma (-1, t)|}
\right)\\
&=
\partial_s \partial_t f (-1, t)
-
\langle
\partial_s \partial_t \gamma (-1, t),
{\bm \tau} (-1, t)
\rangle
\partial_s f (-1, t)\\
&=
\partial_s \partial_t f (-1, t)
+
\partial_s F_m[k] (-1, t)
\cot \theta_-
\partial_s f (-1, t).
\end{align*}
Similarly, we have
\begin{align*}
\partial_t \partial_s f (1, t)
=
\partial_s \partial_t f (1, t)
+
\partial_s F_m[k] (1, t)
\cot \theta_+
\partial_s f (1, t).
\end{align*}
This together with the boundary conditions of \eqref{MP} and \eqref{eq:pm1Gamma} implies that
\begin{align*}
0
&=
\partial_t
\langle
\partial_s \gamma (-1, t),
{\bm v}_{\theta_- + \frac{\pi}{2}}
\rangle\\
&=
\left\langle
-
\frac{\partial_s F_m[k] (-1, t)}
{\sin \theta_-}
{\bm v}_0
+
\partial_s F_m[k] (-1, t)
\cot \theta_-
\partial_s \gamma (-1, t),
{\bm v}_{\theta_- + \frac{\pi}{2}}
\right\rangle\\
&=
\partial_s F_m[k] (-1, t).
\end{align*}
Similarly, we have
\[
\partial_s F_m[k](1, t)=0.
\]
This completes the proof.
\end{proof}


Next, we derive the equation of $k$. For $(u, t) \in (-1, 1) \times [0, T)$, see Lemma \ref{lem:ptk}. When $u=\pm1$ and $t\in[0,T)$, Lemma \ref{lem:pm1commutator} yields the following:


\begin{lemma}\label{lem:pm1kt}
Let $m \in \N \cup \{0\}$, $\theta \in (0,\pi]$ and $\theta_-, \theta_+ \in (0,2\pi)\setminus\{\pi\}$. Assume $\gamma:(-1,1)\times[0,T)\to\R^2$ is a solution to \eqref{MP}. Then the following holds$:$
\begin{align*}
\partial_t k(\pm1,t)
=
\partial_s^2F_m[k](\pm1,t)
\quad
\text{for all}
\quad
t\in[0,T).
\end{align*}
\end{lemma}


\begin{proof}
From \eqref{FS}, \eqref{eq:pm1Gamma}, and Lemma \ref{lem:pm1commutator}, we have
\begin{align*}
\partial_t k(-1,t)
&=
\partial_t
\left\langle
\partial_s^2\gamma(-1,t),
{\bm\nu}_{\theta_-+\frac{\pi}{2}}
\right\rangle
=
\left\langle
-
\frac{\partial_s^2F_m[k](-1,t)}
{\sin\theta_-}
{\bm v}_0,
{\bm v}_{\theta_-+\frac{\pi}{2}}
\right\rangle\\
&=
\partial_s^2F_m[k](-1,t)
\end{align*}
for all $t\in[0,T)$. Similarly, we have
\[
\partial_t k(1,t)
=
\partial_s^2F_m[k](1,t).
\]
This completes the proof.
\end{proof}

From the definition of $G_m[k]$, if the initial curve satisfies \eqref{A1}, then the following lemma holds:
\begin{lemma}\label{lem:u}
Let $m \in \N \cup \{ 0 \}$, $\theta \in (0, \pi]$ and $\theta_-, \theta_+ \in (0, 2\pi) \setminus \{ \pi \}$. Assume $\gamma : (-1,1) \times [0, T) \to \R^2$ is a smooth regular solution to $\eqref{MP}$ with initial curve $\gamma_0$ satisfying \eqref{ACF:BC} and \eqref{A1}. Then $\gamma$ satisfies
  \begin{align}\label{eq:u}
   |\partial_u \gamma| = \frac{L[\gamma]}{2} \quad \text{for all} \quad (u, t) \in (-1, 1) \times [0, T).
  \end{align}
\end{lemma}

\begin{proof}
From \eqref{def:Gmk} and \eqref{eq:dtLGamma}, we have
  \begin{align*}
    \partial_t |\partial_u \gamma|
    &= \frac{\langle \partial_u \gamma, \partial_u \partial_t \gamma \rangle}{|\partial_u \gamma|}
     = (- k F_m[k] + \partial_s G_m[k]) |\partial_u \gamma|
     = \frac{\frac{d}{dt}L[\gamma]}{L[\gamma]} |\partial_u \gamma|,
  \end{align*}
for all $(u, t) \in (-1, 1) \times [0, T)$. Since $\gamma$ is a regular curve, by dividing both sides by $|\partial_u \gamma (u, t)|$, we have
  \begin{align*}
     \frac{\partial_t|\partial_u \gamma|}{|\partial_u \gamma|}
     = \frac{\frac{d}{dt}L[\gamma]}{L[\gamma]},
  \end{align*}
and then,
  \begin{align*}
   \partial_t \log |\partial_u \gamma|
   = \frac{d}{dt} \log L[\gamma],
  \end{align*}
for all $(u, t) \in (-1, 1) \times [0, T)$. Integrating both sides, we have
  \begin{align*}
    \log |\partial_u \gamma|
    &= \log |\partial_u \gamma_0| + \log L[\gamma] - \log L[\gamma_0]\\
    &= \log L[\gamma] + \log \frac{|\partial_u \gamma_0|}{L[\gamma_0]}\\
    &= \log L[\gamma] - \log 2
    = \log \frac{L[\gamma]}{2}.
  \end{align*}
Thus \eqref{eq:u} holds in $(-1, 1) \times [0, T)$. This completes the proof.
\end{proof}


Now, we recall the definition of $E[\gamma]$:

\begin{align*}
E[\gamma]
&=
L[\gamma]
-
\langle \gamma (1), {\bm v}_{\theta + \theta_+} \rangle
+
\langle \gamma (-1), {\bm v}_{\theta_-} \rangle.
\end{align*}
The equation of \eqref{MP} is the $H^{-m}$-gradient flow for $E[\gamma]$ under the boundary conditions of \eqref{MP}. We prove that $E[\gamma]$ is non-increasing and $A[\gamma]$ is preserved for solutions $\gamma$ to the problem \eqref{MP}.

\begin{lemma}\label{lem:dtEdtA}
 Let $m \in \N \cup \{ 0 \}$, $\theta \in (0, \pi]$ and $\theta_-, \theta_+ \in (0, 2\pi) \setminus \{ \pi \}$. Assume $\gamma : (-1,1) \times [0, T) \to \R^2$ is a solution to $\eqref{MP}$. Then,
  \begin{align*}
    \frac{d}{dt} E[\gamma] = - \|\partial_s^m (k - \bar{k}) \|_2^2 \quad \text{and} \quad
    \frac{d}{dt} A[\gamma] = 0 \quad \text{for all} \quad t \in [0, T).
  \end{align*}
\end{lemma}

\begin{proof}
From \eqref{eq:pm1Gamma}, we compute $\frac{d}{dt} \langle \gamma(- 1, t), {\bm v}_{\theta_-} \rangle$ and obtain
  \begin{align*}
    \frac{d}{dt} \langle \gamma(- 1, t), {\bm v}_{\theta_-} \rangle
      &=  \left\langle - \frac{F_m[k] (-1, t)}{\sin \theta_-} {\bm v}_0, {\bm v}_{\theta_-} \right\rangle\\
      &= - F_m[k](- 1, t) \cot \theta_-.
  \end{align*}
Similarly, we have $\frac{d}{dt} \langle \gamma (1, t), {\bm v}_{\theta + \theta_+} \rangle = - F_m [k] (1, t) \cot \theta_+$. This together with the boundary conditions of \eqref{MP} and \eqref{eq:dtLGamma} implies that
\begin{align*}
\frac{d}{dt} E[\gamma]
&=
\frac{d}{dt}
\left(
L[\gamma]
-
\langle \gamma (1, t), {\bm v}_{\theta + \theta_+} \rangle
+
\langle \gamma(-1, t), {\bm v}_{\theta_-} \rangle
\right)\\
&=
-\int_\gamma k F_m[k] \, ds
=
- \|\partial_s^m (k-\bar{k}) \|_2^2.
\end{align*}

Furthermore, the boundary conditions of \eqref{MP} and Lemmas \ref{lem:commutator} and \ref{lem:ptk} imply that

\begin{align*}
\frac{d}{dt}A[\gamma]
&=
-\frac12
\int_\gamma
\Big(
\langle \partial_t\gamma,{\bm\nu}\rangle
+
\langle \gamma,\partial_t{\bm\nu}\rangle
+
\langle \gamma,{\bm\nu}\rangle
(-kF_m[k]+\partial_sG_m[k])
\Big)\,ds\\
&=
-\frac12\int_\gamma F_m[k]\,ds
+\frac12\int_\gamma
\left(
\partial_sF_m[k]\langle\gamma,{\bm\tau}\rangle
\right)\,ds\\
&\quad
+\frac12\int_\gamma
\left(
G_m[k]\langle\gamma,k{\bm\tau}\rangle
\right)\,ds
+\frac12\int_\gamma
\left(
F_m[k]\langle\gamma,k{\bm\nu}\rangle
\right)\,ds\\
&\quad -\frac12\int_\gamma
\left(
\partial_sG_m[k]\langle\gamma,{\bm\nu}\rangle
\right)\,ds\\
&=
-\frac12\int_\gamma F_m[k]\,ds
+\frac12
\left[
F_m[k]\langle\gamma,{\bm\tau}\rangle
\right]_{-1}^1
-\frac12
\left[
G_m[k]\langle\gamma,{\bm\nu}\rangle
\right]_{-1}^1\\
&=
\frac{1}{2}
\left(
F_m[k](1,t)|\gamma(1,t)|\cos\theta_+
-
F_m[k](-1,t)|\gamma(-1,t)|\cos\theta_-
\right)\\
&\quad
-\frac{1}{2}
\left(
G_m[k](1,t)|\gamma(1,t)|\sin\theta_+
-
G_m[k](-1,t)|\gamma(-1,t)|\sin\theta_-
\right)\\
&=0.
\end{align*}

This completes the proof.
\end{proof}


\begin{remark}
From Lemma \ref{lem:dtEdtA}, the solution $\gamma$ to \eqref{MP} preserves $A[\gamma]$. Thus, if we assume $A[\gamma_0] > 0$, from Lemma \ref{lem:Lbdd}, we obtain the following uniform lower bound for $L[\gamma]$:
\begin{align}\label{eq:Llowerbdd}
L[\gamma]^2 \ge 2\theta A[\gamma_0] >0
\quad
\text{in}
\quad
[0,T).
\end{align}
\end{remark}

\section{Asymptotic behavior of \eqref{MP}}\label{asy}

\subsection{Local-in-time solvability}\label{loc}
In this section, we prove the existence and uniqueness of a local-in-time solution to \eqref{MP}.  In Shimojo--Kagaya \cite{SK_2016}, the existence and uniqueness of a local-in-time solution to \eqref{MP} with a concave graph $\gamma_0$ as initial data were studied when $m = 0$. Kagaya \cite{K_2020} subsequently extended the argument to the setting without the assumption that the initial curve $\gamma_0$ is a graph. We prove the local-in-time existence of solutions to \eqref{MP} by an argument similar to theirs.
Assume $\gamma_0 \in C^{2m + 3}((-1, 1); \R^2)$ satisfies \eqref{ACF:BC} and \eqref{A1}. We define $\Theta_0 \in C^{2m + 2}((-1, 1))$ by
  \begin{align*}
   {\bm \tau}_0 (u) = \partial_s \gamma_0 (u) &= (\cos \Theta_0 (u), \sin \Theta_0 (u))
  \end{align*}
and prove the local-in-time solvability of the following problem and uniqueness of the solution:
  \begin{align}\label{prob:Thetaztau}
    \begin{cases}
      \partial_\tau \Theta(z, \tau) = (-1)^m \partial_z^{2m +2} \Theta(z, \tau)\\
        \quad \quad \quad  + \left( P[\Theta](z, \tau) + Q[\Theta](\tau) z \right) \partial_z \Theta (z, \tau) &\quad \text{for} \quad0 < z < 1, \quad 0< \tau,\\
       \Theta(0, \tau) = \theta_-, \quad \Theta(1, \tau) = \theta + \theta_+, &\quad \text{for} \quad 0< \tau,\\
       \partial_z^{2j} \Theta (0, \tau) = \partial_z^{2j} \Theta (1, \tau) = 0 &\quad \text{for} \quad j \in \{ 1, 2, 3, \dots, m \}\\
       \quad &\quad \quad \text{and} \quad 0< \tau \quad \text{if} \quad m \geq 1,\\
      \Theta(z, 0) = \Theta_0 (z) &\quad \text{for} \quad 0< z < 1,
    \end{cases}
  \end{align}
where the parameter $z$ satisfies
 \begin{align*}
    z = z(u) = \frac{u + 1}{2},
  \end{align*}
and $P[\Theta]$ and $Q[\Theta]$ are defined by
  \begin{align*}
    P[\Theta] (z, \tau) &:= (-1)^m \int_0^z \partial_{\tilde{z}} \Theta \partial_{\tilde{z}}^{2m} ( \partial_{\tilde{z}} \Theta - \omega ) \, d \tilde{z}\\
                        & \quad - (-1)^m \left[ \partial_z^{2m} (\partial_z \Theta - \omega ) \right](0, \tau) \cot \theta_-,\\
    Q[\Theta] (\tau) &:= (-1)^{m+1} \int_0^1 \partial_{\tilde{z}} \Theta \partial_{\tilde{z}}^{2m} ( \partial_{\tilde{z}}  \Theta - \omega ) \, d \tilde{z}\\
                   &\quad + (-1)^{m+1} \left[ \partial_z^{2m} (\partial_z \Theta - \omega ) \right](1, \tau) \cot \theta_+\\   
                    &\quad -  (-1)^{m+1} \left[ \partial_z^{2m} (\partial_z \Theta - \omega ) \right](0, \tau) \cot \theta_-.
  \end{align*}
Here, from \eqref{ACF:BC}, $\Theta_0$ satisfies the boundary conditions of \eqref{prob:Thetaztau}. Once we obtain the smooth solution $\Theta$, we can construct the solution to \eqref{MP}. 
We consider the following ODE:
  \begin{align}
    \begin{cases}\label{eq:g}
      \frac{d}{d \tau} h(\tau) &= Q[\Theta] (\tau) \quad \text{for} \quad 0 < \tau,\\
      h(0) &= \log L[\gamma_0].
    \end{cases}
  \end{align}
Using the solution $h (\tau)$ to \eqref{eq:g}, we define $ \mathcal{L}$ by
  \begin{align*}
    \mathcal{L}(\tau) := e^{h(\tau)}.
  \end{align*}
Thus, we construct the solution $\gamma = (\gamma_1, \gamma_2)$ to \eqref{MP} as follows:
  \begin{equation*}
  \begin{aligned}
    \gamma_1 (u, t) &= \gamma_1(-1, 0) 
                                   + (-1)^m \frac{1}{\sin \theta_-} \int_0^t \left( \frac{2}{\mathcal{L}} \right)^{2m} \partial_u^{2m} \left( \frac{2}{\mathcal{L}}\partial_u \Theta(-1, \tilde{t}) - \frac{\omega}{\mathcal{L}} \right) \, d\tilde{t}\\
                                  &\quad + \frac{\mathcal{L}(t)}{2} \int_{-1}^u \cos \Theta \, d\tilde{u},\\
    \gamma_2 (u, t) &= \frac{\mathcal{L}}{2} \int_{-1}^u \sin \Theta \, d\tilde{u},
   \end{aligned}
   \end{equation*}
where $z$ and $\tau$ are related to $u$ and $t$ by
  \begin{align*}
    t = t(\tau) = \int_0^{\tau} e^{(2m + 2) h(\tilde{\tau})} \, d \tilde{\tau},\qquad
   u = u(z) = 2z -1.
  \end{align*}
It follows that $\gamma$ is a solution to \eqref{MP} satisfying $L[\gamma] = \mathcal{L}$ and \eqref{eq:u} in $[0, T)$.
Before proving the local-in-time existence and uniqueness of \eqref{prob:Thetaztau}, we outline the derivation of \eqref{prob:Thetaztau} in Remark \ref{rem:derivation}. We then establish its local-in-time existence and uniqueness. 
\begin{remark}\label{rem:derivation}
We explain the derivation of \eqref{prob:Thetaztau}. Let $\Theta : (-1, 1) \times [0, T) \to \mathbb{R}$ be a smooth angle function. Assume $\gamma : (-1, 1) \times [0, T) \to \mathbb{R}^2$ is the solution to \eqref{MP} with initial curve $\gamma_0 : (-1, 1) \to \mathbb{R}^2$, the tangential vector of $\gamma$ is given by
  \begin{align*}
     {\bm \tau} (u, t) &:= (\cos \Theta (u, t), \sin \Theta (u, t)) \quad \text{for} \quad (u, t) \in (-1, 1) \times [0, T),
  \end{align*}
and $\gamma$ satisfies \eqref{eq:u}.
Thus $\gamma$ and $\Theta$ satisfy
  \begin{align*}
    \partial_s = \frac{1}{|\partial_u \gamma|} \partial_u = \frac{2}{L[\gamma]} \partial_u,\qquad
    \partial_s {\bm \tau} =  \partial_s \Theta {\bm \nu}.
  \end{align*}
Thus, from \eqref{FS}, we have
  \begin{align*}
       \partial_s \Theta &= k.
   \end{align*}
This together with Lemma \ref{lem:ptk} implies that $\Theta : (0, L[\gamma]) \times [0, T) \to \mathbb{R}$ satisfies
  \begin{align*}
   \begin{cases}\label{eq:Theta}
    \partial_t \Theta = (-1)^m \partial_s^{2 m + 2} \Theta\\
    \quad \quad \quad +  \left( P[\Theta] + \frac{s}{L[\gamma]} Q[\Theta] \right) \partial_s \Theta
       \quad &\text{in} \quad (0, L[\gamma]) \times (0, T),\\
    \Theta(0, t) = \theta_-, \quad \Theta(L[\gamma], t) = \theta + \theta_+, \quad &\text{in} \quad (0, T),\\
    \partial_s^{2j} \Theta(0, t) = \partial_s^{2j} \Theta(L[\gamma], t) = 0\\
    \quad \quad \quad \quad \quad \quad \quad \text{for} \quad j \in \{ 1, 2, 3, \dots, m \} \quad &\text{in} \quad (0, T) \quad \text{if} \quad m \geq 1,\\
    \Theta(s, 0) = \Theta_0(s) \quad &\text{in} \quad (0, L[\gamma]),
   \end{cases}
  \end{align*}
with the following free boundary condition:
  \begin{align*}
     \frac{d}{dt} L[\gamma] = Q[\Theta] \quad \text{in} \quad [0, T),
  \end{align*}
where $\Theta_0$ satisfies $\partial_s \gamma_0 = (\cos \Theta_0 , \sin \Theta_0)$ and $P[\Theta]$ and $Q[\Theta]$ are given by
 \begin{equation*}
  \begin{aligned}
    P[\Theta] (s, t) 
      &:= (-1)^m \int_0^s \partial_{\tilde{s}} \Theta \partial_{\tilde{s}}^{2m} (\partial_{\tilde{s}} \Theta - \frac{\omega}{L[\gamma]}) \, d\tilde{s}\\
         &\quad - (-1)^m \partial_s^{2m} (\partial_s \Theta(0, t) - \frac{\omega}{L[\gamma]}) \cot \theta_-,\\
    Q[\Theta] (t) 
      &:= - (-1)^m \int_\gamma \partial_s \Theta \partial_s^{2m} (\partial_s \Theta - \frac{\omega}{L[\gamma]}) \, ds\\
      &\quad - (-1)^m \partial_s^{2m} (\partial_s \Theta (L[\gamma], t) - \frac{\omega}{L[\gamma]}) \cot \theta_+\\
      &\quad + (-1)^m \partial_s^{2m} (\partial_s \Theta (0, t) - \frac{\omega}{L[\gamma]}) \cot \theta_-.
  \end{aligned}
 \end{equation*}
Finally, we normalize the interval $(0, L[\gamma])$ to $(0, 1)$ by the following transformation:
  \begin{align*}
    \tau = \tau(t) = \int_0^t \frac{1}{L[\gamma]^{2m + 2}} \, d\tilde{t}, \qquad
    z = z(s) = \frac{s}{L[\gamma]},
  \end{align*}
and we derive \eqref{prob:Thetaztau}.
\end{remark}
We now prove the following lemma:
  \begin{lemma}\label{lem:Gtau}
 Let $m \in \mathbb{N} \cup \{ 0 \}$, $\theta \in (0, \pi]$ and $\theta_-, \theta_+ \in (0, 2 \pi) \setminus \{ \pi \}$. Assume $\Theta_0 \in C^{2m+2}([0, 1])$ satisfies the boundary conditions of \eqref{prob:Thetaztau}. Then there exists a time $T > 0$ such that the problem \eqref{prob:Thetaztau} has a unique smooth solution
  \begin{align*}
    \Theta \in C^{2m+2, 1}([0,1] \times [0, T)) \cap C^\infty ([0, 1] \times (0, T)).
   \end{align*}
  \end{lemma}
\begin{proof}
We can prove Lemma \ref{lem:Gtau} by a standard argument, so we only outline the proof here.
If $\Theta$ satisfies the boundary conditions of \eqref{prob:Thetaztau}, then there exists $w$ such that
  \begin{align*}
    \partial_z^{2j} w(0, \tau) = \partial_z^{2j} w(1, \tau) = 0 \quad( j \in \{0, 1, 2, \dots, m\} ) &\quad \text{for} \quad 0 < \tau < T
  \end{align*}
and $\Theta$ can be represented by
  \begin{align*}
    \Theta(z, \tau) = w(z, \tau) + x(z)
  \end{align*}
where $x(z) := (\theta + \theta_{+} - \theta_{-})z + \theta_-$. Using this, \eqref{prob:Thetaztau} can be transformed into the following problem:

  \begin{align}
    \begin{cases}\label{eq:w}
      \partial_{\tau} w = (-1)^m \partial_z^{2m+2} w\\
      \quad \quad \quad + (P[w + x] + Q[w + x] z) \partial_z (w + x) &\quad \text{for} \quad 0 < z < 1, 0 < \tau < T,\\
      \partial_z^{2j} w(0, \tau) = \partial_z^{2j} w(1, \tau) = 0 &\quad \text{for}\quad j \in \{0, 1, 2, \dots, m\}\\
      \quad &\quad \quad \quad \text{and} \quad 0 < \tau < T,\\
      w(z, 0) = w_0(z) &\quad \text{for} \quad 0 < z < 1,
    \end{cases}
  \end{align}

where $w_0$ satisfies $\Theta_0 = w_0 + x$. Here, we define the differential operator $A: w \mapsto (-1)^m \partial_z^{2m + 2} w$ with domain 
  \begin{align*}
    D(A) = \{ w \in C^{2m + 2} ([0, 1]) \, | \, \partial_z^{2j} w (0) = \partial_z^{2j} w (1) = 0 \quad \text{for} \quad j \in \{0, 1, 2, \dots, m\} \}.
  \end{align*} 
We can check that $A: D(A) \subset C([0, 1]) \to C([0, 1])$ is sectorial in $C([0, 1])$ and thus generates an analytic semigroup on $C([0, 1])$ (see \cite[Theorem 3.2.5]{L_1995}). Furthermore, since the nonlinear term on the right-hand side is Lipschitz continuous with respect to $w$, we see that \eqref{eq:w} has a unique solution $w$ and $w \in C^{2m + 2 + (2m + 2) \alpha, 1 + \alpha} ([0, 1] \times (0, T))$ for all $\alpha \in (0, 1)$ (see \cite[Proposition 7.3.3]{L_1995}). We can then prove that $w \in C^\infty ([0, 1] \times (0, T))$ by parabolic regularity theory. This completes the proof.
\end{proof}

Using Lemma \ref{lem:Gtau}, we establish the local-in-time existence and uniqueness of a solution to \eqref{MP}.
\begin{theorem}\label{thm:local-in-timesolofMP}
Let $m \in \mathbb{N} \cup \{ 0 \}$, $\theta \in (0, \pi]$ and $\theta_-, \theta_+ \in (0, 2 \pi) \setminus \{ \pi \}$. Assume $\gamma_0 \in C^{2m+ 3}((-1, 1); \R^2)$ satisfies \eqref{ACF:BC} and \eqref{A1}. Then there exists $T > 0$ such that the problem \eqref{MP} has a unique smooth solution 
  \begin{align*}
    \gamma \in C^{2m + 3, 1}([-1, 1] \times [0, T)) \cap C^{\infty} ([-1, 1] \times (0, T)).
  \end{align*}
\end{theorem}
\begin{proof}
By Lemma \ref{lem:u}, every solution to \eqref{MP} satisfies \eqref{eq:u}. From Lemma \ref{lem:Gtau} and the above argument, it is clear that a smooth solution under \eqref{eq:u} can be constructed uniquely. This completes the proof.
\end{proof}


\subsection{Global-in-time solvability and full limit convergence}\label{sec:glo}
In this section, we prove the global-in-time solvability and the convergence of \eqref{MP}. To begin with, we summarize some properties of higher-order derivatives of the curvature of the solution to \eqref{MP}.
We use the symbol $P_N^{l, r}$ for a polynomial with coefficients depending on $\bar{k}$, each monomial of which is of the form
  \begin{align*}
    \prod_{j = 1}^N \partial_s^{p_j} (k - \bar{k}) 
  \end{align*}
with 
  \begin{align*}
    \sum_{j = 1}^N p_j = l, \quad 0 \leq p_j \leq r \quad (j = 1, 2, 3, \dots, N), \quad N \geq 1.
  \end{align*}
Using $P_N^{l, r}[k]$, we have the following lemma:
 \begin{lemma}\label{lem:kslt}
 Let $m \in \N \cup \{ 0 \}$, $\theta \in (0, \pi]$ and $\theta_-, \theta_+ \in (0, 2\pi) \setminus \{ \pi \}$. Assume $\gamma : (-1,1) \times [0, T) \to \R^2$ is a smooth solution to $\eqref{MP}$. Then the following holds$:$
   \begin{align}\label{eq:kslt} 
     \begin{aligned}
       \partial_t \partial^l_s k
       &= (-1)^m \partial^{l+2m+2}_s k
        + \sum_{j = \max \{2m, l \}}^{2m + l} P_3^{2m + l, j} + \sum_{j = \max \{2m, l\}}^{2m + l} P_2^{2m + l, j}\\
        &\quad+ \bar{k}^2 \partial_s^l F_m[k]
            + \partial_s^{l + 1} k G_m[k]
    \end{aligned}
  \end{align} 
for all $l \in \N$ and $(u, t) \in (-1, 1) \times [0, T)$.
\end{lemma}
\begin{proof}
We prove Lemma \ref{lem:kslt} by induction. 
For $l = 1$, we observe from Lemmas~\ref{lem:commutator} and~\ref{lem:ptk} that 
  \begin{align}\label{eq:kst}
    \begin{aligned} 
    \partial_t \partial_s k 
      &= \partial_s \partial_t k + \left( k F_m[k] - \partial_s G_m[k] \right) \partial_s k\\
      &= (-1)^m \partial^{2m+3}_s k
           + 3 k \partial_s k F_m[k]
           + k^2 \partial_s F_m[k] + \partial_s^2 k G_m[k].
    \end{aligned}
  \end{align}
Thus \eqref{eq:kslt} holds for $l=1$. 
Assume $\eqref{eq:kslt}$ for $l \in \N$ with $l \leq n$.  
It follows from Lemmas~\ref{lem:commutator} and \ref{lem:ptk} that 
  \begin{align*}
    \partial_t \partial_s^{n+1} k
      &= \partial_s \partial_t \partial^n_s k
      + \left( k F_m[k] - \partial_s G_m[k] \right) \partial^{n+1}_s k\\
      &= \partial_s \Bigg[ (-1)^m \partial^{2m+n+2}_s k
           + \sum_{j = \max \{2m, n \}}^{2m + n} P_3^{2m + n, j}\\
           &\quad  \quad \quad \quad + \sum_{j = \max \{2m, n\}}^{2m + n} P_2^{2m + n, j}
            + \bar{k}^2 \partial_s^n F_m[k]
            + \partial_s^{n + 1} k G_m[k]
\Bigg]\\
     &\quad+ k F_m[k] \partial^{n+1}_s k
      - \partial_s^{n + 1} k \partial_s G_m[k]\\
      &= (-1)^m \partial^{n+2m+3}_s k
        + \sum_{j = \max \{2m, n+1\}}^{2m + n + 1} P_3^{2m + n + 1, j}\\
        &\quad+ \sum_{j = \max \{2m, n+1\}}^{2m + n + 1} P_2^{2m + n + 1, j}
        + \bar{k}^2 \partial_s^{n + 1} F_m[k] + \partial_s^{n + 2} k G_m[k].
  \end{align*}
Thus $\eqref{eq:kslt}$ holds for $l = n +1$. This completes the proof. 
\end{proof}

\begin{lemma}\label{lem:pm1kslt}
Let $m \in \N \cup \{ 0 \}$, $\theta \in (0, \pi]$ and $\theta_-, \theta_+ \in (0, 2\pi) \setminus \{ \pi \}$. Assume $\gamma : (-1,1) \times [0, T) \to \R^2$ is a smooth solution to $\eqref{MP}$. Then the following holds$:$
  \begin{align*}
    \partial_t \partial_s^l k (\pm1, t) = \partial_s^{l + 2} F_m[k] (\pm1, t)
  \end{align*}
for all $l \in \N \cup \{ 0 \}$ and $t \in [0, T)$. 
\end{lemma}
\begin{proof}
From Lemmas \ref{lem:pm1commutator} and \ref{lem:pm1kt}, we have
  \begin{align*}
    \partial_t \partial_s^l k (\pm1, t) 
      = \partial_s^l \partial_t k (\pm1, t)
      = \partial_s^{l + 2} F_m[k] (\pm1, t).
  \end{align*}
This completes the proof.
\end{proof}

Using the boundary condition of \eqref{MP}, Lemmas \ref{lem:pm1commutator} and \ref{lem:pm1kslt}, we find the boundary condition on higher-order derivatives of the curvature of the solution to \eqref{MP}. 

\begin{lemma}\label{lem:ks2l+1}
Let $m \in \N \cup \{ 0 \}$, $\theta \in (0, \pi]$ and $\theta_-, \theta_+ \in (0, 2\pi) \setminus \{ \pi \}$. Assume $\gamma : (-1,1) \times [0, T) \to \R^2$ is a smooth solution to $\eqref{MP}$. Then the following holds$:$
\begin{equation}\label{eq:ks2l+1}
\partial^{2l+1}_s k (\pm1, t) = 0
\end{equation}
for all $l \in \N \cup \{ 0 \}$ and $t \in [0, T)$. 
\end{lemma}
\begin{proof}
We prove Lemma \ref{lem:ks2l+1} by induction. 
When $l \leq m - 1$, \eqref{eq:ks2l+1} follows from the boundary conditions of \eqref{MP}. When $l = m$, \eqref{eq:ks2l+1} has been proved in the proof of Lemma \ref{lem:pm1commutator}. Assume \eqref{eq:ks2l+1} holds for $l \in \N$ with $m \leq l \leq n$. Thus, from Lemma \ref{lem:pm1kslt}, we have
  \begin{align*}
    0 &= \partial_t \partial_s^{2n - 2m + 1} k (\pm1, t)
       = \partial_s^{2n - 2m + 1} \partial_t k (\pm1, t)\\
       &= \partial_s^{2n - 2m + 1} \partial_s^2 F_m[k] (\pm1, t)
       = (-1)^m \partial_s^{2n + 3} k (\pm1, t).
  \end{align*}
Thus \eqref{eq:ks2l+1} holds for $l = n + 1$. This completes the proof. 
\end{proof}

Here, we recall the definition of $\mathcal{E}_0 [\gamma]$ and $\mathcal{E}_1 [\gamma]$:
\begin{align*}
\begin{aligned}
\mathcal{E}_0 [\gamma] := E [\gamma] \int (k - \bar{k})^2 \, ds,\qquad
\mathcal{E}_1 [\gamma] := E [\gamma]^3 \int_\gamma ( \partial_s k )^2 \, ds.
\end{aligned}
\end{align*}

Using Lemma \ref{lem:ks2l+1}, we calculate the time derivative of $\mathcal{E}_0 [\gamma]$ and $\mathcal{E}_1 [\gamma]$ as follows:

\begin{lemma}\label{lem:dtmathcalE0}
Let $m = 0$, $\theta \in (0, \pi]$ and $\theta_\pm \in (0, 2\pi) \setminus \{ \pi \}$. Assume $\gamma : (-1,1) \times [0, T) \to \mathbb{R}^2$ is the smooth solution to $\eqref{MP}$. Then the following holds$:$
    \begin{align}
    \begin{aligned}\label{eq:dtmathcalE0}
    \frac{d}{dt} \mathcal{E}_0 [\gamma]
    &= \frac{d}{dt} E [\gamma] \int_\gamma (k - \bar{k})^2 \, ds
            - 2 E [\gamma] \| \partial_s k \|_2^2
           + E [\gamma] \int_\gamma (k - \bar{k})^4\, ds\\
    &\quad+ 3 \bar{k} E [\gamma] \int_\gamma (k - \bar{k})^3 \, ds
               + 2 \bar{k}^2 E [\gamma] \| k - \bar{k} \|_2^2\\
     &\quad + 2 E [\gamma] \int_\gamma (k - \bar{k}) \partial_s k G_0 [k] \, ds
                + E [\gamma] \int_\gamma (k - \bar{k})^2 \partial_s G_0 [k] \, ds
    \end{aligned}
  \end{align}
 for all $t \in [0, T).$
\end{lemma}

\begin{proof}
It follows from Lemma \ref{lem:rot-invariant} that
  \begin{align*}
    \frac{d}{dt} \bar{k} = - \frac{\frac{d}{dt} L[\gamma]}{L[\gamma]} \bar{k}.
  \end{align*}
This together with Lemmas \ref{lem:commutator}, \ref{lem:ptk}, and \ref{lem:rot-invariant} implies that
  \begin{align*}
   \frac{d}{dt} \mathcal{E}_0 [\gamma]
    &= \frac{d}{dt} E [\gamma] \int_\gamma (k - \bar{k})^2 \, ds
       + 2 E [\gamma] \int_\gamma (k - \bar{k}) \partial_t (k - \bar{k}) \, ds\\
       &\quad + E [\gamma] \int_\gamma (k - \bar{k})^2 ( - k F_0 [k] + \partial_s G_0 [k]) \, ds\\
    &= \frac{d}{dt} E [\gamma] \int_\gamma (k - \bar{k})^2 \, ds
         + 2 E [\gamma] \int_\gamma (k - \bar{k}) \partial_s^2 F_0 [k] \, ds\\
        &\quad + 2 E [\gamma] \int_\gamma (k - \bar{k}) k^2 F_0 [k] \, ds\\
    &\quad + 2 E [\gamma] \int_\gamma (k - \bar{k}) \partial_s k G_0 [k] \, ds
         - E [\gamma] \int_\gamma (k - \bar{k})^2 k F_0 [k] \, ds\\
        &\quad + E [\gamma] \int_\gamma (k - \bar{k})^2 \partial_s G_0 [k] \, ds.
   \end{align*}
Using Lemma \ref{lem:ks2l+1} and integrating by parts, we have
  \begin{align*}
    \frac{d}{dt} \mathcal{E}_0 [\gamma]
    & = \frac{d}{dt} E [\gamma] \int_\gamma (k - \bar{k})^2 \, ds
            - 2 E [\gamma] \| \partial_s k \|_2^2
           + E [\gamma] \int_\gamma (k - \bar{k})^3 F_0 [k] \, ds\\
    &\quad + 3 \bar{k} E [\gamma] \int_\gamma (k - \bar{k})^2 F_0 [k] \, ds
           + 2 \bar{k}^2 E [\gamma] \| k - \bar{k} \|_2^2 \\
    &\quad + 2 E [\gamma] \int_\gamma (k - \bar{k}) \partial_s k G_0 [k] \, ds
           + E [\gamma] \int_\gamma (k - \bar{k})^2 \partial_s G_0 [k] \, ds,
  \end{align*}
where we used $\int_\gamma (k -\bar{k})\, ds = 0$. Thus we have \eqref{eq:dtmathcalE0}.
\end{proof}

\begin{lemma}\label{lem:dtmathcalE1}
Let $m \in \mathbb{N}$, $\theta \in (0, \pi]$ and $\theta_-, \theta_+ \in (0, 2\pi) \setminus \{ \pi \}$. Assume $\gamma : (-1,1) \times [0, T) \to \mathbb{R}^2$ is a smooth solution to $\eqref{MP}$. Then the following holds$:$
  \begin{align*}
    \begin{aligned}
    \frac{d}{dt} \mathcal{E}_1 [\gamma]
        &= 3 E[\gamma]^2 \frac{d}{dt} E [\gamma] \int_\gamma (\partial_s k)^2 \, ds
              - E[\gamma]^3 \| \partial_s^{m + 2} k \|_2^2\\
       &\quad + 5 E[\gamma]^3 \int_\gamma (k - \bar{k}) (\partial_s k)^2 F_m[k] \, ds
        +5 \bar{k} E[\gamma]^3 \int_\gamma (\partial_s k)^2 F_m[k] \, ds\\
         &\quad + 2 E[\gamma]^3 \int_\gamma \partial_s k (k - \bar{k})^2 \partial_s F_m[k] \, ds\\
        &\quad + 4 \bar{k} E[\gamma]^3 \int_\gamma \partial_s k (k - \bar{k}) \partial_s F_m[k] \, ds
         + 2 \bar{k}^2 E[\gamma]^3 \| \partial_s^{m + 1} k \|_2^2
     \end{aligned}
   \end{align*}
 for all $t \in [0, T)$.
\end{lemma}

\begin{proof}
From Lemmas \ref{lem:commutator} and \ref{lem:ptk} and \eqref{eq:kst}, we have
  \begin{align*}
    \frac{d}{dt} \mathcal{E}_1 [\gamma]
      &= 3 E[\gamma]^2 \frac{d}{dt} E [\gamma] \int_\gamma (\partial_s k)^2 \, ds
            + 2 E[\gamma]^3 \int_\gamma \partial_s k (\partial_t \partial_s k) \, ds\\
      &\quad + E[\gamma]^3 \int_\gamma (\partial_s k)^2 (- k F_m[k] + \partial_s G_m [k]) \, ds\\
      &= 3 E[\gamma]^2 \frac{d}{dt} E [\gamma] \int_\gamma (\partial_s k)^2 \, ds
        + (-1)^m 2 E[\gamma]^3 \int_\gamma \partial_s k \partial^{2m+3}_s k \, ds\\
        &\quad + 6 E[\gamma]^3 \int_\gamma  k (\partial_s k)^2 F_m[k] \, ds
         + 2 E[\gamma]^3 \int_\gamma k^2 \partial_s k \partial_s F_m[k] \, ds\\
        &\quad + 2 E[\gamma]^3 \int_\gamma \partial_s k \partial_s^2 k G_m[k] \, ds\\
        &\quad + E[\gamma]^3 \int_\gamma (\partial_s k)^2 (- k F_m[k] + \partial_s G_m [k]) \, ds\\
      &= 3 E[\gamma]^2 \frac{d}{dt} E [\gamma] \int_\gamma (\partial_s k)^2 \, ds
           + (-1)^m 2 E[\gamma]^3 \int_\gamma \partial_s k \partial_s^{2m + 3} k \, ds\\
      &\quad + 5 E[\gamma]^3 \int_\gamma k ( \partial_s k)^2 F_m[k] \, ds
            + 2 E[\gamma]^3 \int_\gamma k^2 \partial_s k \partial_s F_m [k] \, ds.
  \end{align*}
Using Lemma \ref{lem:ks2l+1} and integrating by parts, we obtain
  \begin{align*}
    \begin{aligned}
      \frac{d}{dt} \mathcal{E}_1 [\gamma]
        &= 3 E[\gamma]^2 \frac{d}{dt} E [\gamma] \int_\gamma (\partial_s k)^2 \, ds
              - E[\gamma]^3 \| \partial_s^{m + 2} k \|_2^2\\
       &\quad + 5 E[\gamma]^3 \int_\gamma (k - \bar{k}) (\partial_s k)^2 F_m[k] \, ds
        +5 \bar{k} E[\gamma]^3 \int_\gamma (\partial_s k)^2 F_m[k] \, ds\\
         &\quad + 2 E[\gamma]^3 \int_\gamma \partial_s k (k - \bar{k})^2 \partial_s F_m[k] \, ds\\
        &\quad + 4 \bar{k} E[\gamma]^3 \int_\gamma \partial_s k (k - \bar{k}) \partial_s F_m[k] \, ds
         + 2 \bar{k}^2 E[\gamma]^3 \| \partial_s^{m + 1} k \|_2^2
    \end{aligned}
  \end{align*}
where we used $\int_\gamma (k -\bar{k})\, ds = 0$. This completes the proof.
\end{proof}
From the definition of $F_m[k]$ and the boundary conditions in \eqref{P}, we can characterize the stationary solutions of \eqref{P} as follows.
\begin{lemma}\label{lem:stationarysol}
Every smooth stationary solution $\gamma$ to \eqref{P} is either a straight line segment or a circular arc.
\end{lemma}
\begin{proof}
Let $\gamma$ be a smooth stationary solution to \eqref{P}. Since $F_m[k] \equiv 0$, it follows from the boundary conditions in \eqref{P} that
\begin{align*}
0 = \int_\gamma k F_m[k] \, ds = \| \partial_s^m (k - \bar{k}) \|_2^2.
\end{align*}
Hence,
\begin{align*}
\partial_s^m (k - \bar{k}) \equiv 0.
\end{align*}
If $m = 0$, this immediately yields $k \equiv \bar{k}$. Assume $m \geq 1$, and set $f:= k - \bar{k}$. Then $\partial_s^m f \equiv 0$, and therefore $f$ is a polynomial in the arc-length parameter $s$ of degree at most $m - 1$. Since $\bar{k}$ is independent of $s$, the boundary conditions in \eqref{P} imply
\[
\partial_s^{2j - 1} f(0) = \partial_s^{2j - 1} f(L[\gamma]) = 0 \quad \text{for} \quad j \in \{1, 2, 3, \dots, m\}.
\]
We now show that $f$ is constant. If $m = 1$, then $\partial_s f \equiv 0$, and hence $f$ is constant. Suppose that $m \geq 2$, and let $2l-1$ be the largest odd integer strictly less than $m$. Since $f$ has degree at most $m-1$, $\partial_s^{2l-1}f$ is a polynomial of degree at most one. Moreover,
\[
\partial_s^{2l - 1}f (0) = \partial_s^{2l - 1}f (L[\gamma]) = 0,
\]
and hence
\[
\partial_s^{2l - 1} f \equiv 0.
\]
It follows that $f$ has degree at most $2l - 2$. Repeating the same argument for the lower odd-order derivatives, we successively obtain
\[
\partial_s^{2l - 3} f \equiv 0, \quad \partial_s^{2l - 5} f \equiv 0,\quad \dots, \partial_s f \equiv 0.
\]
Therefore, $f$ is constant. On the other hand, by the definition of $\bar{k}$, we have $\int_\gamma f \, ds = 0$. Hence, $f \equiv 0$, and consequently
\[
k \equiv \bar{k}.
\]
Thus, $\gamma$ has constant curvature and is therefore either a straight line segment or a circular arc.
\end{proof}
In what follows, we assume that $(\theta, \theta_-, \theta_+) \in \mathcal{D}$, where
\[
\mathcal{D} := \left\{ (\sigma, \alpha, \beta) \,\middle| \, \sigma \in (0, \pi), \, \alpha, \beta \in (0, 2\pi) \setminus \{ \pi \}, \, | \sigma + \beta - \alpha | < \pi, \, \mu (\sigma, \alpha, \beta) < 1 \right\},
\]
with
\[
\mu (\sigma, \alpha, \beta) := \sqrt{\frac{\cos^2 \alpha + \cos^2 \beta - 2 \cos \sigma \cos \alpha \cos \beta}{\sin^2 \sigma}}.
\]
Let $T_{max} > 0$ be the maximal existence time of the solution to \eqref{MP} and assume $T_{max} < \infty$. To prove the global-in-time solvability of \eqref{MP}, we first prove $E[\gamma] > 0$ by establishing the following lemma, which plays a crucial role in the subsequent analysis.
\begin{lemma}\label{lem:LleqE}
Let $\theta \in (0, \pi)$ and $\theta_\pm \in (0, 2\pi) \setminus \{ \pi \}$. Assume $\gamma : (-1,1) \to \mathbb{R}^2$ is a smooth curve satisfying $\gamma (-1) \in \eta_0$, $\gamma (1) \in \eta_\theta$, ${\bm \tau} (-1) = {\bm v}_{\theta_-}$, and ${\bm \tau} (1) = {\bm v}_{\theta + \theta_+}$. Then
  \begin{align*}
    E[\gamma] \geq (1 - \mu) L[\gamma].
  \end{align*}
Consequently, if $(\theta, \theta_-, \theta_+) \in \mathcal{D}$, then 
  \begin{align*}
    E[\gamma] > 0 \quad \text{and} \quad L[\gamma] \leq \frac{E[\gamma]}{1 - \mu}.
  \end{align*}
\end{lemma}
\begin{proof}
Since $\gamma (-1) \in \eta_0$ and $\gamma (1) \in \eta_\theta$, there exist constants $r_\pm \in \mathbb{R}$ such that  
\[
\gamma(-1) = r_- {\bm v}_0, \quad \gamma(1) = r_+ {\bm v}_\theta.
\]
Set 
\begin{align*}
{\bm d} 
  &:= \gamma(1) - \gamma (-1) = r_+ {\bm v}_\theta - r_- {\bm v}_0.
\end{align*}
Since the Euclidean distance between the end points is bounded by the length of the curve, we have
\begin{align*}
|{\bm d}| \leq L[\gamma].
\end{align*}
Furthermore, from the definition of $E[\gamma]$, we have
\begin{align*}
  E[\gamma] 
    &= L [\gamma] - \langle \gamma(1), {\bm v}_{\theta + \theta_+} \rangle + \langle \gamma(-1), {\bm v}_{\theta_-} \rangle
      = L [\gamma] - r_+ \cos \theta_+ + r_- \cos \theta_-\\
    &= L [\gamma] - \cos \theta_- (r_+ \cos \theta - r_-) + \frac{\cos \theta \cos \theta_- - \cos \theta_+}{\sin \theta} r_+ \sin \theta\\
    &= L [\gamma] + \left\langle 
         \left( - \cos \theta_-, \frac{\cos \theta \cos \theta_- - \cos \theta_+}{\sin \theta} \right), {\bm d}
          \right\rangle.
\end{align*}
Moreover,
\begin{align*}
\left| \left( - \cos \theta_-, \frac{\cos \theta \cos \theta_- - \cos \theta_+}{\sin \theta} \right) \right|^2 
&= \frac{\cos^2 \theta_- + \cos^2 \theta_+ - 2 \cos \theta \cos \theta_- \cos \theta_+}{\sin^2 \theta}\\ 
&= \mu^2.
\end{align*}
Therefore, by the Cauchy--Schwarz inequality, we have
\begin{align*}
E[\gamma] \geq L[\gamma] - \mu |{\bm d}| \geq (1 - \mu) L[\gamma].
\end{align*}
Finally, since $(\theta, \theta_-, \theta_+) \in \mathcal{D}$, we have $\mu < 1$. Consequently,
\[
E[\gamma] > 0 \quad \text{and} \quad L[\gamma] \leq \frac{E[\gamma]}{1 - \mu}.
\]
\end{proof}
 Combining Lemmas \ref{lem:stationarysol} and \ref{lem:LleqE}, we can further characterize the stationary solutions when $(\theta, \theta_-, \theta_+) \in \mathcal{D}$.
\begin{corollary}\label{cor:k>0arc}
Let $(\theta, \theta_-, \theta_+) \in \mathcal{D}$, and let $\gamma$ be a smooth stationary solution to \eqref{P} satisfying $A[\gamma] > 0$. Then $\gamma$ is a circular arc with positive curvature. 
\end{corollary}
\begin{proof}
From Lemma \ref{lem:stationarysol}, we see that $\gamma$ is either a straight line segment or a circular arc. If $\gamma$ is a straight line segment, then its tangent vector is constant, and hence
$v_{\theta_-}=v_{\theta+\theta_+}.$ Therefore, $\theta+\theta_+-\theta_-\in 2\pi\mathbb Z$. Since $(\theta,\theta_-, \theta_+)\in\mathcal D$, we have $|\theta+\theta_+-\theta_-|<\pi$, and consequently $\omega[\gamma]=\theta+\theta_+-\theta_-=0$. Thus we have
\begin{align*}
\mu 
 &= \sqrt{\frac{(\cos \theta_- - \cos \theta \cos \theta_+)^2 - \cos^2 \theta \cos^2 \theta_+ + \cos^2 \theta_+}{\sin^2 \theta}}\\
 &= \sqrt{\frac{(- \sin \theta \sin \theta_+)^2 + \sin^2 \theta \cos^2 \theta_+}{\sin^2 \theta}} =1.
\end{align*}
This contradicts $\mu < 1$. Hence, $\gamma$ is a circular arc with nonzero constant curvature $\tilde{k}$. Moreover, from \eqref{FS}, we have
\begin{align*}
  E[\gamma] 
    &= L[\gamma] - \langle \gamma(1), {\bm v}_{\theta + \theta_+} \rangle + \langle \gamma(-1), {\bm v}_{\theta_-} \rangle\\
    &=L[\gamma] - \int_\gamma \partial_s \langle \gamma, {\bm \tau} \rangle \, ds
    = - \int_\gamma \tilde{k} \langle \gamma, {\bm \nu} \rangle \, ds = 2 \tilde{k} A[\gamma].
\end{align*}
Since $A[\gamma] > 0$ and $E[\gamma] > 0$ by Lemma \ref{lem:LleqE}, it follows that
\[
\tilde{k} = \frac{E[\gamma]}{2 A[\gamma]} > 0.
\]
Therefore, $\gamma$ is a circular arc with positive curvature.
\end{proof}
Using Lemmas \ref{lem:dtmathcalE0}, \ref{lem:dtmathcalE1}, and \ref{lem:LleqE}, we prove the uniform boundedness of $\mathcal{E}_0 [\gamma]$ and $\mathcal{E}_1[\gamma]$ in the following lemma:

\begin{lemma}\label{lem:mathcalE0bdd}
Let $m = 0$ and $(\theta, \theta_-, \theta_+) \in \mathcal{D}$. Then there exists $K^* (m, \theta, \theta_\pm) > 0$ such that the following holds$:$
Assume $\gamma : (-1,1) \times [0, T_{max}) \to \mathbb{R}^2$ is the smooth solution to $\eqref{MP}$ with the initial curve $\gamma_0$ satisfying \eqref{ACF:BC}, \eqref{A1}, \eqref{A2} and \eqref{A3}. Then
\begin{align*}
\mathcal{E}_{0} [\gamma] < K^* \quad \text{for all} \quad t \in [0, T_{max}),
\end{align*}
where $2K^*$ is a unique positive root of
\begin{align}\label{def:g0}
\begin{aligned}
g_0 (x) &:= 2\left(1 - \frac{\omega^2}{\pi^2} \right) - \frac{8}{\pi (1 - \mu)} x
- C_0(\theta, \theta_\pm)\sqrt{x}
= 0,
\end{aligned}
\end{align}
where
\begin{align}
&\quad C_0 (\theta, \theta_\pm)\notag\\
&:= \left( \left( \sqrt{\frac{2}{\pi}} + \frac{4}{\pi} \right) |\omega| + \left( 2 + \frac{1}{\pi} \right) |\cot \theta_+| + \left( 4 + \frac{1}{\pi} \right) |\cot \theta_-| \right) \sqrt{\frac{2}{\pi (1 - \mu)}} \label{eq:C_0}.
\end{align}
\end{lemma}
\begin{proof}
We prove Lemma \ref{lem:mathcalE0bdd} by contradiction. Assume that there exists $T^* \in (0,T_{max})$ such that $\mathcal{E}_0 [\gamma](t) < K^*$ for $t \in [0,T^*)$ and
$\mathcal{E}_0 [\gamma](T^*) = K^*$. From Lemmas \ref{lem:dtmathcalE0} and \ref{lem:LleqE}, we have
\begin{align*}
\frac{d}{dt} \mathcal{E}_0 [\gamma]
&\leq - 2 E [\gamma] \|\partial_s k \|_2^2
+ E [\gamma] \int_\gamma (k - \bar{k})^4 \, ds
+ 3 \bar{k} E [\gamma] \int_\gamma (k - \bar{k})^3 \, ds\\
&\quad + 2 \bar{k}^2 E [\gamma] \| k - \bar{k} \|_2^2
+ 2 E [\gamma] \int_\gamma (k - \bar{k}) \partial_s k G_0[k] \, ds\\
&\quad + E [\gamma] \int_\gamma (k - \bar{k})^2 \partial_s G_0[k] \, ds\\
&=: - 2 E [\gamma] \| \partial_s k \|_2^2 + I_1 + I_2 + I_3 + I_4 + I_5.
\end{align*}
Using Lemmas \ref{lem:Poincare1}, \ref{lem:LleqE}, and H{\"o}lder's inequality, we have
\begin{align*}
I_1 &\leq E [\gamma] \| k - \bar{k} \|_\infty^2 \| k -\bar{k} \|_2^2
\leq \left( \frac{ 2 L[\gamma]}{\pi} \right) \mathcal{E}_0 [\gamma] \| \partial_s k \|_2^2\\
&\leq \frac{2}{\pi(1- \mu)} \mathcal{E}_0 [\gamma] E[\gamma] \| \partial_s k \|_2^2,\\
I_2 &\leq 3 |\bar{k}| E [\gamma] \| k - \bar{k} \|_\infty^2 \left( L[\gamma]^{\frac{1}{2}} \| k - \bar{k} \|_2 \right)\\
&\leq \frac{3}{\pi} \sqrt{\frac{2}{\pi (1 - \mu)}} |\omega| \sqrt{ \mathcal{E}_0 [\gamma]} E[\gamma] \| \partial_s k \|_2^2.
\end{align*}
From Lemmas \ref{lem:Poincare1} and \ref{lem:LleqE}, we have
\begin{align*}
I_3 &\leq 2 \frac{\omega^2}{\pi^2} E [\gamma] \| \partial_s k \|_2^2.
\end{align*}
Here, we observe from the definition of $G_m[k]$ that
\begin{align*}
G_0[k]
&= \int_0^s (k - \bar{k})^2 \, d\tilde{s}
+ \bar{k} \int_0^s (k - \bar{k}) \, d\tilde{s}
+ \frac{\frac{d}{dt} L[\gamma]}{L[\gamma]} s
- (k (-1, t) - \bar{k}) \cot \theta_-,\\
\partial_s G_0 [k]
&= (k - \bar{k})^2
+ \bar{k} (k - \bar{k})
+ \frac{\frac{d}{dt} L[\gamma]}{L[\gamma]}.
\end{align*}
This together with \eqref{eq:dtLGamma}, Lemmas \ref{lem:Poincare1}, \ref{lem:LleqE}, and H{\"o}lder's inequality implies that
\begin{align*}
&\quad I_4\\
&\leq 2 E [\gamma] \| \partial_s k \|_2 \| k - \bar{k} \|_2^3
+ 2 E [\gamma] |\bar{k}| \left( L[\gamma]^{\frac{1}{2}} \| k - \bar{k} \|_2 \right)
\| k - \bar{k} \|_2 \| \partial_s k \|_2\\
&\quad + 2 E [\gamma] \left( \| k - \bar{k} \|_2^2 + \| k - \bar{k} \|_\infty |\cot \theta_+|
+ \| k - \bar{k} \|_\infty | \cot \theta_- |\right) \| k - \bar{k} \|_2 \| \partial_s k \|_2\\
&\quad + 2 E[\gamma] |\cot \theta_-| \| k - \bar{k} \|_\infty \| k - \bar{k} \|_2 \| \partial_s k \|_2\\
&\leq \frac{4}{\pi (1 - \mu)} \mathcal{E}_0 [\gamma] E [\gamma] \| \partial_s k \|_2^2
+ \frac{2 |\omega|}{\pi \sqrt{1 - \mu}} \sqrt{ \mathcal{E}_0 [\gamma]} E [\gamma] \| \partial_s k \|_2^2\\
&\quad + 2\sqrt{\frac{2}{\pi (1 - \mu)}} |\cot \theta_+| \sqrt{ \mathcal{E}_0 [\gamma] } E[\gamma] \| \partial_s k \|_2^2\\
&\quad + 4 \sqrt{\frac{2}{\pi (1 - \mu)}} |\cot \theta_-| \sqrt{ \mathcal{E}_0 [\gamma]} E[\gamma] \| \partial_s k \|_2^2,
\end{align*}
and
\begin{align*}
I_5 &\leq E[\gamma] \| k - \bar{k} \|_\infty^2 \| k - \bar{k} \|_2^2
+ E [\gamma] |\bar{k}| \| k - \bar{k} \|_\infty \| k - \bar{k} \|_2^2
- \frac{E [\gamma]}{L[\gamma]} \| k - \bar{k} \|_2^4\\
&\quad + \frac{E [\gamma]}{L [\gamma]} |\cot \theta_+| \| k - \bar{k} \|_\infty \| k - \bar{k} \|_2^2
+ \frac{E [\gamma]}{L [\gamma]} |\cot \theta_-| \| k - \bar{k} \|_\infty \| k - \bar{k} \|_2^2\\
&\leq \frac{2}{\pi (1 - \mu)} \mathcal{E}_0 E[\gamma] \| \partial_s k \|_2^2 + \frac{1}{\pi} \sqrt{\frac{2}{\pi (1 - \mu)}} |\omega| \sqrt{\mathcal{E}_0} E[\gamma] \| \partial_s k \|_2^2\\
&\quad + \frac{1}{ \pi} \sqrt{\frac{2}{\pi (1 - \mu)}} |\cot \theta_+| \sqrt{ \mathcal{E}_0 [\gamma]} E [\gamma] \| \partial_s k \|_2^2\\
&\quad + \frac{1}{\pi} \sqrt{\frac{2}{\pi (1 - \mu)}} |\cot \theta_-| \sqrt{ \mathcal{E}_0 [\gamma]} E [\gamma] \| \partial_s k \|_2^2.
\end{align*}
Combining the above, we have
\begin{equation*}
\begin{aligned}
\frac{d}{dt} \mathcal{E}_0 [\gamma]
&\leq - \left[ 2\left(1 - \frac{\omega^2}{\pi^2} \right) - \frac{8}{\pi (1 - \mu)} \mathcal{E}_0 [\gamma]
- C_0(\theta, \theta_\pm)\sqrt{ \mathcal{E}_0 [\gamma] } \right] E[\gamma] \| \partial_s k \|_2^2,
\end{aligned}
\end{equation*}
where $C_0 (\theta, \theta_\pm)$ is defined by \eqref{eq:C_0}.
From Lemma \ref{lem:LleqE} and \eqref{def:g0}, we define $\delta_0 := g_0 (K^*) > 0$ and we have
\begin{align}\label{eq:step1}
\frac{d}{dt} \mathcal{E}_0 [\gamma]
\leq - \delta_0 E[\gamma] \| \partial_s k \|_2^2 \leq 0
\end{align}
for all $t \in [0, T^*]$. This contradicts the assumption that $\mathcal{E}_0 [\gamma] (T^*) = K^*$ and this completes the proof.
\end{proof}
Here we define
\[
N_n^l := \left\{(p_1, p_2, p_3, \dots , p_n) \,\middle|\, p_1, p_2, p_3, \dots , p_n \in \mathbb{N} \cup \{ 0 \}, \,\, \sum_{j= 1}^n p_j = l \right\}.
\]
Using it, we prove the following lemma:

\begin{lemma}\label{lem:mathcalE1bdd}
Let $m \in \N$ and $(\theta, \theta_-, \theta_+) \in \mathcal{D}$. Then there exists $K^* (m, \theta, \theta_\pm) > 0$ such that the following holds$:$ Assume $\gamma : (-1,1) \times [0, T_{max}) \to \R^2$ is the smooth solution to \eqref{MP} with the initial curve $\gamma_0$ satisfying \eqref{ACF:BC}, \eqref{A1}, \eqref{A2} and \eqref{A3}. Then
  \begin{align*}
    \mathcal{E}_1 [\gamma] < K^* \quad \text{for all} \quad t \in [0, T_{max}),
  \end{align*}
where $2K^*$ is a unique positive root of
   \begin{align}\label{def:gm}
    \begin{aligned}
    g_m (x) &:= 2\Big(1 - \frac{\omega^2}{\pi^2} \Big) - \Bigg( \sum_{N_3^{m - 1}} \frac{(m - 1)!}{p_1! p_2! p_3!} A_{p_1 p_2 p_3} \Bigg) \frac{5}{\pi^3 (1 - \mu)^3} x\\
    &\quad - \Bigg( \sum_{N_2^{m - 1}} \frac{(m - 1)!}{p_1! p_2!} B_{p_1 p_2} \Bigg) \frac{5 |\omega|}{\pi^\frac{3}{2} (1 - \mu)^{\frac{3}{2}}} \sqrt{x}\\
    &\quad - \Bigg( \sum_{N_3^m} \frac{m !}{p_1! p_2! p_3!} C_{p_1 p_2 p_3} \Bigg) \frac{2}{\pi^3 (1 - \mu)^3} x\\
    &\quad - \Bigg( \sum_{N_2^m} \frac{m !}{p_1! p_2!} D_{p_1 p_2} \Bigg) \frac{4 |\omega|}{\pi^\frac{3}{2} (1 - \mu)^{\frac{3}{2}}} \sqrt{x}
    = 0,
   \end{aligned}
  \end{align}
with constants $A_{p_1 p_2 p_3}, B_{p_1 p_2}, C_{p_1 p_2 p_3}, D_{p_1 p_2} > 0$ depending only on $p_1$, $p_2$ and $p_3$.
\end{lemma}

\begin{proof}
For notational simplicity, we write $\bar{K} := k - \bar{k}$ through out the proof. We prove Lemma \ref{lem:mathcalE1bdd} by contradiction. Assume that there exists $T^* \in (0,T_{max})$ such that $\mathcal{E}_1 [\gamma](t) <~K^*$ for $t \in [0,T^*)$ and 
$\mathcal{E}_1 [\gamma](T^*) = K^*$. From Lemmas \ref{lem:dtmathcalE1} and \ref{lem:LleqE}, we have
  \begin{align*}
    \frac{d}{dt} \mathcal{E}_1 [\gamma]  
       &\leq - 2 E [\gamma]^3 \| \partial_s^{m + 2} k \|_2^2
         + 5 E [\gamma]^3 \int_\gamma \bar{K} ( \partial_s k)^2 F_m[k] \, ds\\
      &\quad + 5 \bar{k} E [\gamma]^3 \int_\gamma ( \partial_s k)^2 F_m[k] \, ds
        + 2 E [\gamma]^3 \int_\gamma \bar{K}^2 \partial_s k \partial_s F_m[k] \, ds\\
      &\quad + 4 \bar{k} E [\gamma]^3 \int_\gamma \bar{K} \partial_s k \partial_s F_m[k] \, ds
       + 2 \bar{k}^2 E [\gamma]^3 \| \partial_s^{m + 1} k \|_2^2\\
      &=: - 2 E [\gamma]^3 \| \partial_s^{m+2} k \|_2^2 + I_1 + I_2 + I_3 + I_4 + I_5.
  \end{align*}
Using Lemmas \ref{lem:Poincare1}, \ref{lem:Poincare2}, \ref{lem:k-interpolation}, \ref{lem:ks2l+1}, and \ref{lem:LleqE}, we have
  \begin{align*}
    I_1 &= - 5 E [\gamma]^3 \int_\gamma \partial_s^{m - 1} \left[ \bar{K}(\partial_s k)^2 \right] \partial_s^{m + 1}\bar{K} \, ds\\
          &= - 5 E [\gamma]^3 \sum_{N_3^{m - 1}} \int_\gamma \frac{(m - 1)!}{p_1! p_2! p_3!} \partial_s^{p_1} \bar{K} \partial_s^{p_2 + 1} \bar{K}
          \partial_s^{p_3 + 1} \bar{K} \partial_s^{m + 1} \bar{K} \, ds\\
          &\leq 5 E [\gamma]^3 \sum_{N_3^{m - 1}} \Bigg( \frac{(m - 1)!}{p_1! p_2! p_3!} \| \partial_s^{p_1} \bar{K} \|_8 \| \partial_s^{p_2 + 1} \bar{K} \|_8 
          \| \partial_s^{p_3 + 1} \bar{K} \|_4 \| \partial_s^{m +1} \bar{K} \|_2 \Bigg)\\
          &\leq 5 E [\gamma]^3 \left( \sum_{N_3^{m - 1}} \frac{(m - 1)!}{p_1! p_2! p_3!} A_{p_1 p_2 p_3} \right) \| \bar{K} \|_2^2
             \| \partial_s^{m+1} k \|_2 \| \partial_s^{m + 2} k \|_2\\
          &\leq \left( \sum_{N_3^{m -1}} \frac{(m - 1)!}{p_1! p_2! p_3!} A_{p_1 p_2 p_3} \right) \frac{5}{\pi^3(1 - \mu)^3} \mathcal{E}_1 [\gamma] E [\gamma]^3 \| \partial_s^{m+2} k \|_2^2,
  \end{align*}
  \begin{align*}
    I_2 &= - 5 \bar{k} E [\gamma]^3 \int_\gamma \partial_s^{m -1} ( \partial_s \bar{K} )^2 \partial_s^{m + 1} \bar{K} \, ds\\
          &= - 5 \bar{k} E [\gamma]^3 \sum_{N_2^{m- 1}} \int_\gamma \frac{(m - 1)!}{p_1! p_2!} \partial_s^{p_1 + 1} \bar{K}
                   \partial_s^{p_2 + 1} \bar{K} \partial_s^{m + 1} \bar{K} \, ds\\
          &\leq 5 |\bar{k}| E [\gamma]^3 \sum_{N_2^{m - 1}} \Bigg( \frac{(m - 1)!}{p_1! p_2!} \| \partial_s^{p_1 + 1} \bar{K} \|_4 \| \partial_s^{p_2 + 1} \bar{K} \|_4 \| \partial_s^{m + 1} \bar{K} \|_2 \Bigg)\\
          &\leq 5 |\bar{k}| E [\gamma]^3 \left( \sum_{N_2^{m - 1}} \frac{(m - 1)!}{p_1! p_2!} B_{p_1p_2} \right) \| \bar{K} \|_2^{1+ \alpha_m} \| \partial_s^{m+2} k\|_2^{1- \alpha_m} \| \partial_s^{m + 1} k\|_2\\
          &\leq \left( \sum_{N_2^{m - 1}} \frac{(m - 1)!}{p_1! p_2!} B_{p_1 p_2} \right) \frac{5 |\omega|}{\pi^\frac{3}{2} (1 - \mu)^{\frac{3}{2}}} \sqrt{\mathcal{E}_1 [\gamma]} E [\gamma]^3 \| \partial_s^{m+2} k \|_2^2,
  \end{align*}
  \begin{align*}
    I_3 &= 2 E[\gamma]^3 \int_\gamma \partial_s^{m} \left(\bar{K}^2 \partial_s k \right) \partial_s^{m + 1} \bar{K} \, ds\\
          & = 2 E [\gamma]^3 \sum_{N_3^m} \int_\gamma \frac{m !}{p_1! p_2! p_3!} \partial_s^{p_1} \bar{K} \partial_s^{p_2}\bar{K} \partial_s^{p_3 + 1} \bar{K} \partial_s^{m + 1} \bar{K} \, ds\\
          &\leq 2 E [\gamma]^3 \sum_{N_3^m} \Bigg( \frac{m !}{p_1! p_2! p_3!} \| \partial_s^{p_1} \bar{K} \|_8 \| \partial_s^{p_2} \bar{K} \|_8 \| \partial_s^{p_3 + 1} \bar{K} \|_4 \| \partial_s^{m + 1} \bar{K} \|_2 \Bigg)\\
         &\leq 2 E [\gamma]^3 \left( \sum_{N_3^m} \frac{m !}{p_1! p_2! p_3!} C_{p_1 p_2 p_3} \right) \|\bar{K}\|_2^2 \| \partial_s^{m+1} k \|_2 \| \partial_s^{m+2} k \|_2\\
         &\leq \left( \sum_{N_3^m} \frac{m !}{p_1! p_2! p_3!} C_{p_1 p_2 p_3} \right) \frac{2}{\pi^3 (1 - \mu)^3} \mathcal{E}_1 [\gamma] E[\gamma]^3 \| \partial_s^{m + 2} k \|_2^2,
  \end{align*}
and
  \begin{align*}
    I_4 &= 4 \bar{k} E[\gamma]^3 \int_\gamma \partial_s^m \left( \bar{K} \partial_s \bar{K} \right) \partial_s^{m + 1} \bar{K} \, ds\\
          &= 4 \bar{k} E [\gamma]^3 \sum_{(p_1, p_2) \in N_2^m} \frac{m !}{p_1! p_2!} \int_\gamma \partial_s^{p_1} \bar{K} \partial_s^{p_2 +1} \bar{K} \partial_s^{m + 1} \bar{K} \, ds\\
          &\leq 4 |\bar{k}| E [\gamma]^3 \Bigg( \sum_{N_2^m} \frac{m !}{p_1! p_2!} \| \partial_s^{p_1} \bar{K} \|_4 \| \partial_s^{p_2 + 1} \bar{K} \|_4 \| \partial_s^{m + 1} \bar{K} \|_2 \Bigg)\\
          &\leq 4 |\bar{k}| E [\gamma]^3 \left( \sum_{N_2^m} \frac{m !}{p_1! p_2!} D_{p_1 p_2} \right) \| \bar{K} \|_2^{1 + \alpha_m} \| \partial_s^{m + 2} k \|_2^{1 - \alpha_m} \| \partial_s^{m + 1} k \|_2\\
          &\leq \left( \sum_{N_2^m} \frac{m !}{p_1! p_2!} D_{p_1 p_2} \right) \frac{4 |\omega|}{\pi^\frac{3}{2} (1 - \mu)^{\frac{3}{2}}} \sqrt{ \mathcal{E}_1 [\gamma]} E [\gamma]^3 \| \partial_s^{m + 2} k \|_2^2,
  \end{align*}
where $\alpha_m := \frac{1}{2(m+2)}$ and $A_{p_1 p_2 p_3}$, $B_{p_1 p_2}$, $C_{p_1 p_2 p_3}$ and $D_{p_1 p_2}$ are positive constants depending only on $p_1$, $p_2$ and $p_3$. Furthermore, from Lemmas \ref{lem:Poincare1}, \ref{lem:Poincare2} and \ref{lem:LleqE}, we have
  \begin{align*}
    I_5 \leq 2 \frac{\omega^2}{\pi^2} E [\gamma]^3 \| \partial_s^{m + 2} k \|_2^2.
  \end{align*}
Combining them, we have
  \begin{align*}
    &\quad \frac{d}{dt} \mathcal{E}_1 [\gamma]\\
    &\leq - \left[ 2\left(1 - \frac{\omega^2}{\pi^2} \right) - \left( \sum_{N_3^{m - 1}} \frac{(m - 1)!}{p_1! p_2! p_3!} A_{p_1 p_2 p_3} \right) \frac{5}{\pi^3 (1 - \mu)^3} \mathcal{E}_1 [\gamma] \right. \\
    &\quad\quad\quad - \left( \sum_{N_2^{m - 1}} \frac{(m - 1)!}{p_1! p_2!} B_{p_1 p_2} \right) \frac{5 |\omega|}{\pi^\frac{3}{2} (1 - \mu)^{\frac{3}{2}}} \sqrt{\mathcal{E}_1 [\gamma]}\\
    &\quad \quad\quad - \left( \sum_{N_3^m} \frac{m !}{p_1! p_2! p_3!} C_{p_1 p_2 p_3} \right) \frac{2}{\pi^3 (1 - \mu)^3} \mathcal{E}_1 [\gamma]\\
    &\quad \quad\quad \left. - \left( \sum_{N_2^m} \frac{m !}{p_1! p_2!} D_{p_1 p_2} \right) \frac{4 |\omega|}{\pi^\frac{3}{2} (1 - \mu)^{\frac{3}{2}}} \sqrt{ \mathcal{E}_1 [\gamma]} \right] E[\gamma]^3 \| \partial_s^{m+2} k \|_2^2.
  \end{align*}
From Lemma \ref{lem:LleqE} and \eqref{def:gm}, we define $\delta_m := g_m (K^*) > 0$ and we have
  \begin{align}\label{eq:step2}
    \frac{d}{dt} \mathcal{E}_1 [\gamma]
      \leq - \delta_m E[\gamma]^3 \| \partial_s^{m + 2} k \|_2^2 \leq 0
  \end{align}
for all $t \in [0, T^*]$. This contradicts the assumption that $\mathcal{E}_1[\gamma] (T^*) = K^*$ and this completes the proof.
\end{proof}

Here we define 
\begin{align*}
  \Phi [\gamma] := 
    \begin{cases}
      \mathcal{E}_0 [\gamma]  \quad \text{if} \quad m = 0,\\
      \mathcal{E}_1 [\gamma]  \quad \text{if} \quad m \geq 1,
    \end{cases}
\end{align*}
and then prove the uniform boundedness of the derivatives of $k$ in the following lemma:
\begin{lemma}\label{lem:pslk}
Let $m \in \mathbb{N} \cup \{ 0 \}$ and $(\theta, \theta_-, \theta_+) \in \mathcal{D}$. Then there exists $K^* (m, \theta, \theta_\pm) > 0$ such that the following holds$:$ Assume $\gamma : (-1,1) \times [0, T_{max}) \to \mathbb{R}^2$ is the smooth solution to $\eqref{MP}$ with the initial curve $\gamma_0$ satisfying \eqref{ACF:BC}, \eqref{A1}, \eqref{A2} and \eqref{A3}. Then 
\begin{align*}
 \|\partial_s^l k\|_2 < C \quad \text{for all} \quad l \in \mathbb{N} \cup \{ 0 \}, \quad t \in [0, T_{max}),
\end{align*}
with a constant $C (m, \theta, \theta_\pm, \gamma_0, l) > 0$.
\end{lemma}
\begin{proof}
When $l = 0$, it follows from \eqref{eq:Llowerbdd} and Lemmas \ref{lem:Poincare1}, \ref{lem:rot-invariant}, \ref{lem:LleqE}, \ref{lem:mathcalE0bdd}, and \ref{lem:mathcalE1bdd} that there exists a constant $C(m, \theta, \theta_\pm, \gamma_0) > 0$ such that
\begin{align*}
  \| k \|_2^2 
    &= \| k - \bar{k} \|_2^2 + \frac{\omega^2}{L[\gamma]}  
    \leq C (\Phi [\gamma] + 1)
    \leq C,
\end{align*}
for all $t \in [0, T_{max})$. When $l \geq 1$, it suffices to consider the case $l = 2n - 1$ with $n \in \mathbb{N}$. Once the estimate is proved for $l = 2n - 1$, the case $l = 2n$ follows from Lemma \ref{lem:Poincare1}. Here Lemmas \ref{lem:kslt} and \ref{lem:ks2l+1} imply that
\begin{align*}
  &\quad \frac{d}{dt} \|\partial_s^{2n - 1} k \|_2^2\\
  &= 2 \int_\gamma \partial_t \partial_s^{2n - 1} k \partial_s^{2n - 1} k \, ds
       + \int_\gamma (\partial_s^{2n - 1} k)^2 (- k F_m[k] + \partial_s G_m [k]) \, ds\\
  &= (-1)^m 2 \int_\gamma \partial_s^{2m + 2n + 1} k \partial_s^{2n - 1} k \, ds\\
  &\quad+ 2 \sum_{j = \max \{2m, 2n-1 \}}^{2m + 2n -1} \left( \int_\gamma \partial_s^{2n -1} k P_3^{2m + 2n -1, j} \, ds 
  + \int_\gamma \partial_s^{2n -1} k P_2^{2m + 2n -1, j} \, ds \right)\\
  &\quad + 2 \bar{k}^2 \int_\gamma \partial_s^{2n -1} \partial_s^{2n -1} F_m[k] \, ds\\
  &= -2 \| \partial_s^{m + 2n} k \|_2^2
   + \int_\gamma P_4^{2m + 4n - 2, m + 2n - 1} \, ds + \int_\gamma P_3^{2m + 4n - 2, m + 2n - 1}\, ds\\
   &\quad + 2 \bar{k}^2 \| \partial_s^{m + 2n - 1} (k - \bar{k}) \|_2^2,
\end{align*}
where we integrated by parts in the last equality. Fix $\varepsilon > 0$ arbitrarily. For $p_1, p_2, p_3, p_4$ with $\sum_{j = 1}^4 p_j = 2m + 4n - 2$ and $0 \leq p_j \leq m + 2n - 1$, we observe from Lemmas \ref{lem:k-interpolation}, \ref{lem:mathcalE0bdd} and \ref{lem:mathcalE1bdd} that
\begin{align*}
\left| \int_\gamma \prod_{j = 1}^4 \partial_s^{p_j} (k - \bar{k}) \, ds \right| 
&\leq \left( \prod_{j= 1}^3 \| \partial_s^{p_j} (k - \bar{k} ) \|_{2^j} \right) \| \partial_s^{p_4} (k - \bar{k}) \|_8\\
&\leq \prod_{j = 1}^4 \| k - \bar{k} \|_2^{1 - \alpha_j} \| \partial_s^{n + 2m} k \|_2^{\alpha_j}\\
&\leq  C \| k - \bar{k} \|_2^{2 + \frac{1}{m + 2n}} \| \partial_s^{m + 2n} k \|_2^{2 - \frac{1}{m + 2n}}\\
&\leq \varepsilon \| \partial_s^{m + 2n} k \|_2^2 + C_\varepsilon \Phi[\gamma]^{4m + 8n + 2} \leq \varepsilon \| \partial_s^{m + 2n} k \|_2^2 + C_\varepsilon
\end{align*}
with a constant $C_\varepsilon (m, \theta, \theta_\pm, \gamma_0, \varepsilon) > 0$ where we used Young's inequality on the last estimate, and
\begin{align*}
  \alpha_j = \frac{1}{m + 2n} \left( p_j + \frac{1}{4} \right), \quad \sum_{j = 1}^4 \alpha_j = 2 - \frac{1}{m + 2n}.
\end{align*}
Similarly, for $q_1, q_2, q_3$ with $\sum_{j = 1}^3 q_j = 2m + 4n -2$ and $0 \leq q_j \leq m + 2n - 1$, we obtain
\begin{align*}
  \left| \int_\gamma \prod_{j = 1}^3 \partial_s^{q_j} (k - \bar{k}) \, ds \right| 
  &\leq \left( \prod_{j = 1}^2 \| \partial_s^{q_j} (k - \bar{k}) \|_{2^j} \right) \| \partial_s^{q_3} (k - \bar{k}) \|_4\\
  &\leq \prod_{j = 1}^3 \| k - \bar{k} \|_2^{1 -\beta_j} \| \partial_s^{m + 2n} k \|_2^{\beta_j}\\
  &\leq C \| k - \bar{k} \|_2^{1 + \frac{3}{2(m + 2n)}} \| \partial_s^{m + 2n} k \|_2^{2 - \frac{3}{2 (m + 2n)}}\\
  &\leq \varepsilon \| \partial_s^{m + 2n} k \|_2^2 + C_\varepsilon \Phi[\gamma]^{2m + 4n +3}\\
   &\leq \varepsilon \| \partial_s^{m + 2n} k \|_2^2 + C_\varepsilon
\end{align*}
with a constant $C_\varepsilon (m, \theta, \theta_\pm, \gamma_0, \varepsilon) > 0$ where
\begin{align*}
\beta_1 &= \frac{1}{m + 2n} q_1, \quad \beta_j = \frac{1}{m + 2n} \left( q_j + \frac{1}{4} \right) \quad \text{for} \quad j = 2, 3,\\\quad \sum_{j = 1}^3 \beta_j &= 2 - \frac{3}{2 (m + 2n)}.
\end{align*}
Moreover, Lemmas \ref{lem:Poincare1} and \ref{lem:Poincare2} imply that
\begin{align*}
  2 \bar{k}^2 \| \partial_s^{m + 2n -1} (k - \bar{k}) \|_2^2 \leq 2 \frac{\omega^2}{\pi^2} \| \partial_s^{m + 2n} k \|_2^2.
\end{align*}
Taking $\varepsilon$ small enough, we have
\begin{align*}
  \frac{d}{dt} \| \partial_s^{2n - 1} k \|_2^2
  \leq - \left( 1 - \frac{\omega^2}{\pi^2} \right) \| \partial_s^{2n +m} k \|_2^2 + C_1,
\end{align*}
with a constant $C_1 (m, \theta, \theta_\pm, \gamma_0) > 0$.
This together with Lemmas \ref{lem:Poincare1}, \ref{lem:Poincare2}, \ref{lem:rot-invariant}, \ref{lem:dtEdtA}, and \ref{lem:LleqE} implies that
\begin{align*}
  \frac{d}{dt} \| \partial_s^{2n - 1} k \|_2^2 \leq - C_2 \| \partial_s^{2n - 1} k \|_2^2 + C_1,
\end{align*}
with constants $C_i (m, \theta, \theta_\pm, \gamma_0, n) > 0$ $(i = 1, 2)$. This clearly implies that
\begin{align*}
  \| \partial_s^{2n - 1} k (t) \|_2^2 
    &\leq \| \partial_s^{2n - 1} k (0) \|_2^2 + \frac{C_1}{C_2}
    \leq C \quad \text{for all} \quad t \in [0, T_{max}),
\end{align*}
with a constant $C (m, \theta, \theta_\pm, \gamma_0, n) > 0$. This completes the proof.
\end{proof}
Using \eqref{eq:Llowerbdd}, Lemmas \ref{lem:dtEdtA}, \ref{lem:LleqE} and \ref{lem:pslk}, we derive the following lemma:

\begin{lemma}\label{lem:pulGmbdd}
Let $m \in \mathbb{N} \cup \{ 0 \}$ and $(\theta, \theta_-, \theta_+) \in \mathcal{D}$. Then there exists $K^* (m, \theta, \theta_\pm) > 0$ such that the following holds$:$ Assume $\gamma : (-1,1) \times [0, T_{max}) \to \mathbb{R}^2$ is the smooth solution to $\eqref{MP}$ with the initial curve $\gamma_0$ satisfying \eqref{ACF:BC}, \eqref{A1}, \eqref{A2} and \eqref{A3}. Then there exist constants $0 < C_1 < C_2$ such that
  \begin{align} \label{eq:pulGmbdd1}
    C_1 \leq  |\partial_u \gamma (u, t)| \leq C_2 \quad \text{for all} \quad (u, t) \in (-1, 1) \times [0, T_{max}).
  \end{align}
Furthermore, for every $l \in \mathbb{N}$, there exists a constant $C_3 > 0$ such that
  \begin{align}\label{eq:pulGmbdd2}
    \|\partial_u^l \gamma \|_\infty \leq C_3 \quad \text{for all} \quad t \in [0, T_{max}).
  \end{align}
\end{lemma}
\begin{proof}
From \eqref{A2}, \eqref{eq:Llowerbdd} and Lemma \ref{lem:LleqE}, we have
  \begin{align*}
    \frac{\sqrt{2 \theta A[\gamma_0]}}{2} \leq |\partial_u \gamma| \leq \frac{E[\gamma_0]}{2(1 - \mu)} \quad \text{for all} \quad (u, t) \in (-1, 1) \times [0, T_{max}).
  \end{align*}
Thus \eqref{eq:pulGmbdd1} holds. We prove \eqref{eq:pulGmbdd2}. From \eqref{A2} and \eqref{eq:Llowerbdd}, we have
  \begin{align}\label{eq:pulGmbdd3}
    \partial_u^l \xi = \left( \frac{L[\gamma]}{2} \right)^l \partial_s^l \xi
  \end{align}
for any smooth vector field $\xi$. When $l = 1$, the assertion follows from \eqref{eq:pulGmbdd1}. When $l \geq 2$, we proceed by induction. From \eqref{FS}, \eqref{eq:Llowerbdd}, \eqref{eq:pulGmbdd3}, Lemmas \ref{lem:dtEdtA} and \ref{lem:LleqE}, we have
  \begin{align*}
    |\partial_u^2 \gamma|
      \leq \left( \frac{L[\gamma]}{2} \right)^2 | k {\bm \nu} |
      \leq C \quad \text{for all} \quad (u, t) \in (-1, 1) \times [0, T_{max})
  \end{align*}
with some constant $C(m, \theta, \theta_\pm, \gamma_0) > 0$ where we used the fact that
  \begin{align}\label{eq:kinfty}
    \| k \|_\infty \leq \| k - \bar{k} \|_\infty + |\bar{k}| \leq \left( \frac{2L[\gamma]}{\pi} \right)^{\frac{1}{2}} \| \partial_s k \|_2 + \frac{|\omega|}{L[\gamma]} \leq C.
  \end{align}
Thus, \eqref{eq:pulGmbdd2} holds for $l = 2$.
Assume that \eqref{eq:pulGmbdd2} holds for $l \in \mathbb{N}$ with $2 \leq l \leq n$. From \eqref{FS}, \eqref{eq:pulGmbdd3}, \eqref{eq:kinfty}, Lemmas \ref{lem:Poincare1}, \ref{lem:Poincare2}, \ref{lem:dtEdtA}, \ref{lem:pslk} and \ref{lem:LleqE}, we have
  \begin{align*}
   | \partial_u^{n + 1} \gamma | = \left( \frac{L[\gamma]}{2} \right)^{n + 1} | \partial_s^{n - 1}( k \bnu) | \leq C \quad \text{for all} \quad (u, t) \in (-1, 1) \times [0, T_{max})
  \end{align*}
with some constant $C(m, \theta, \theta_\pm, \gamma_0, n) > 0$. Thus \eqref{eq:pulGmbdd2} holds. This completes the proof.
\end{proof}

We prove global-in-time solvability of \eqref{MP}. 
\begin{theorem}\label{thm:global-in-timesolofMP}
Let $m \in \mathbb{N} \cup \{ 0 \}$ and $(\theta, \theta_-, \theta_+) \in \mathcal{D}$. Then there exists $K^* (m, \theta, \theta_\pm) > 0$ such that the following holds$:$ 
Assume $\gamma_0$ is a smooth regular curve satisfying \eqref{ACF:BC}, \eqref{A1}, \eqref{A2} and \eqref{A3}. Then problem~\eqref{MP} possesses a smooth global-in-time solution $\gamma : (-1, 1) \times [0, \infty) \to \mathbb{R}^2$.
\end{theorem}
\begin{proof}
We prove it by contradiction. Assume $T_{max} < \infty$. Then, thanks to Lemmas \ref{lem:Lbdd}, \ref{lem:dtEdtA}, \ref{lem:LleqE}, and \ref{lem:pulGmbdd}, we see that
  \begin{align*}
    \limsup_{t \nearrow T_{max}} \| \gamma \|_{H_{du}^{2m + 2}} = \infty.
  \end{align*}
Then, there exists a monotonically increasing sequence $\{ t_j \}_j$ with $t_j \nearrow T_{max}$ as $j \to \infty$ such that
  \begin{align*}
    \| \gamma (t_j) \|_{H_{du}^{2m + 2}} \to \infty \quad \text{as} \quad j \to \infty.
  \end{align*}
From \eqref{eq:kinfty} and Lemmas \ref{lem:Poincare1}, \ref{lem:dtEdtA}, \ref{lem:LleqE} and \ref{lem:pslk}, we have
  \begin{align*}
    &\quad |\gamma (u, t) - \gamma_0(u)|\\
    &\leq \int_0^t |\partial_{\tilde{t}} \gamma| \, d\tilde{t}
    \leq C \left( \int_0^t | F_m[k] | \, d \tilde{t} + \int_0^t |G_m[k]| \, d\tilde{t} \right)\\
    &\leq T_{max} \left(1
       + 2 E[\gamma_0] \| k \|_\infty  + (|\cot \theta_+| + 2|\cot \theta_-|) \right) \| F_m[k] \|_\infty\\
    &\leq C T_{max} \| \partial_s^{2m + 1} k \|_2 \leq C.
  \end{align*}
This together with Lemma \ref{lem:pulGmbdd} leads to a contradiction. This completes the proof.
\end{proof}
In what follows, we prove the full limit convergence of the solution to \eqref{MP}. For this purpose, we establish the exponential decay of $\Phi[\gamma]$ and of the derivatives of $k$.
\begin{lemma}\label{lem:Phiexpdecay}
Let $m \in \mathbb{N} \cup \{ 0 \}$ and $(\theta, \theta_-, \theta_+) \in \mathcal{D}$. Then there exists $K^* (m, \theta, \theta_\pm) > 0$ such that the following holds$:$ Assume $\gamma : (-1,1) \times [0, \infty) \to \mathbb{R}^2$ is the smooth solution to $\eqref{MP}$ with the initial curve $\gamma_0$ satisfying \eqref{ACF:BC}, \eqref{A1}, \eqref{A2} and \eqref{A3}. Then there exists a constant $\delta (m, \theta, \theta_\pm, \gamma_0) > 0$ such that
  \begin{align}\label{eq:Phiexpdecay1}
    \Phi [\gamma] < K^* e^{- \delta t} \quad \text{for all} \quad t \in [0, \infty).
  \end{align}
Moreover, for each $l \in \N$, we find $C(m, \theta, \theta_\pm, \gamma_0, l) > 0$ such that
  \begin{align}\label{eq:Phiexpdecay2}
    \| \partial_s^l k \|_2 \leq C e^{- \delta t} \quad \text{for all} \quad t \in [0, \infty).
  \end{align}
\end{lemma}
\begin{proof}
From Lemmas \ref{lem:mathcalE0bdd} and \ref{lem:mathcalE1bdd}, \eqref{eq:step1} and \eqref{eq:step2} hold for all $t \in [0, \infty)$. Thus, from Lemmas \ref{lem:Poincare1}, \ref{lem:Poincare2}, \ref{lem:dtEdtA} and \ref{lem:LleqE}, we have
  \begin{align*}
    \frac{d}{dt} \Phi [\gamma]
      \leq - \delta_m \left( \frac{\pi}{L[\gamma]} \right)^{2m + 2}\Phi [\gamma]
      \leq - \delta_m \left( \frac{\pi(1 - \mu)}{E[\gamma_0]} \right)^{2m + 2} \Phi [\gamma]
  \end{align*}
for all $t \in [0, \infty)$. From Gronwall's inequality, we have
  \begin{align*}
    \Phi [\gamma]
      \leq \Phi [\gamma_0] \exp \Big( - \delta_m \left( \frac{\pi(1 - \mu)}{E[\gamma_0]} \right)^{2m + 2} t \Big)
      < K^* e^{- \delta t}
  \end{align*}
for all $t \in [0, \infty)$ with some constant $\delta (m, \theta, \theta_\pm, \gamma_0) > 0$. Thus \eqref{eq:Phiexpdecay1} holds.
From \eqref{eq:Llowerbdd}, \eqref{eq:Phiexpdecay1}, Lemmas \ref{lem:Poincare1}, \ref{lem:ks2l+1}, \ref{lem:LleqE}, \ref{lem:pslk}, and H{\"o}lder's inequality, we have
  \begin{align*}
    \| \partial_s^l k \|_2^2
    &\leq \| \partial_s^{2l} k \|_2 \| k - \bar{k} \|_2
    \leq \frac{C}{\sqrt{E[\gamma]}} \sqrt{\Phi[\gamma]}
    \leq \frac{C\sqrt{K^*}}{\sqrt{(1 - \mu)L[\gamma]}} e^{-\frac{\delta}{2} t}\\
    &\leq \frac{C\sqrt{K^*}}{(2\theta (1 - \mu) A[\gamma_0])^{\frac{1}{4}}}e^{-\frac{\delta}{2} t},
  \end{align*}
with some constants $C(m, \theta, \theta_\pm, \gamma_0, l) > 0$ and $\delta (m, \theta, \theta_\pm, \gamma_0) > 0$. Taking square roots and relabeling the positive constants, we obtain \eqref{eq:Phiexpdecay2}.
\end{proof}
At the end of this section, we prove the convergence of \eqref{MP}.

\begin{theorem}\label{thm:convergenceofMP}
Let $m \in \N \cup \{ 0 \}$ and $(\theta, \theta_-, \theta_+) \in \mathcal{D}$. Then there exists $K^* (m, \theta, \theta_\pm) > 0$ such that the following holds$:$
Assume $\gamma_0$ is a smooth regular curve satisfying \eqref{ACF:BC}, \eqref{A1}, \eqref{A2} and \eqref{A3}. Then the unique global-in-time solution $\gamma (\cdot, t)$ to \eqref{MP} converges exponentially to the circular arc $\gamma_\infty (\cdot)$  having rotation number $\omega[\gamma_\infty] = \theta + \theta_+ - \theta_-$ and area $A[\gamma_\infty] = A[\gamma_0]$ as $t \to \infty$ in the $C^\infty$-topology.
\end{theorem}
\begin{proof}
We divide the proof into three steps.

\noindent
{\bf Step 1.}{ \sl We prove subconvergence of the global-in-time solution to \eqref{MP}.} From Lemmas \ref{lem:Poincare1}, \ref{lem:Poincare2}, and \ref{lem:Phiexpdecay}, we have
  \begin{align}\label{eq:convergenceofMP1}
   \begin{aligned}
    \int_0^t \| \partial_t \gamma \|_2 \, d \tilde{t}
      &\leq \int_0^t ( \| F_m[k] \|_2 + \| G_m[k] \|_2 ) \, d\tilde{t}
      \leq C \int_0^t  \| \partial_s^{2m + 1} k \|_2 \, d\tilde{t}\\
      &\leq C \int_0^t e^{- \delta \tilde{t}} \, d\tilde{t}
      \leq \frac{C}{\delta} (1 - e^{- \delta t})
      \leq \frac{C}{\delta},
     \end{aligned}
  \end{align}
for all $t \in [0, \infty)$, where $\delta (m, \theta, \theta_\pm, \gamma_0), C(m, \theta, \theta_\pm, \gamma_0) > 0$. Thus we have
  \begin{align}\label{eq:convergenceofMP2}
    \int_0^\infty \| \partial_t \gamma \|_2 \, dt \leq \frac{C}{\delta}.
  \end{align}
Then there exists a monotonically increasing sequence $\{ t_j \}_j$ with $t_j \nearrow \infty$ such that
  \begin{align*}
    \| \partial_t \gamma (t_j) \|_2 \to 0 \quad \text{as} \quad j \to \infty.
  \end{align*}
By Lemma \ref{lem:pulGmbdd}, there exist constants $0 < C_1 < C_2$ and, for every $l \in \N$, a constant $C_3 > 0$ such that
  \begin{align*}
    C_1 \leq  |\partial_u \gamma (u, t)| \leq C_2, \quad  \| \partial_u^l \gamma \|_\infty \leq C_3, \quad \text{for all} \quad (u, t) \in (-1, 1) \times [0, \infty).
  \end{align*}
Furthermore, by the same argument as in \eqref{eq:convergenceofMP1}, we have
  \begin{align*}
    |\gamma (u, t) - \gamma_0 (u)| \leq \int_0^t | \partial_t \gamma (u, \tilde{t}) | \, d\tilde{t} < C \quad \text{for all} \quad (u, t) \in (-1, 1) \times [0, \infty),
  \end{align*}
where $C (m, \theta, \theta_\pm, \gamma_0) > 0$. This shows that $\{ \gamma (t_j) \}_j$ is uniformly bounded in $C^\infty ([-1, 1])$. Then, using the Arzel{\`a}--Ascoli theorem and the diagonal argument, we obtain a subsequence converging to a curve $\gamma_\infty$ which satisfies the boundary conditions of \eqref{MP} in the $C^\infty$-topology. Moreover, by the definition of $\{t_j\}_j$, $\gamma_\infty$ is an equilibrium of \eqref{MP}.

\noindent
{\bf Step 2.}{ \sl We prove the full limit convergence of the global-in-time solution to \eqref{MP}.}
We define
  \begin{align*}
    \mathcal{G} [\gamma] := \int_{-1}^1 |\gamma (u, t) - \gamma_\infty (u)|^2 \, du.
  \end{align*}
By the same argument as in \cite{AMWW_2020} (see also \cite[Theorem 5.8]{HO_2025}), from \eqref{eq:convergenceofMP2}, we have
  \begin{align}\label{eq:Gconvergence}
    \mathcal{G}[\gamma (t)] \to 0 \quad \text{as} \quad t \to \infty.
  \end{align}
This implies that
 \begin{align*}
   \gamma (\cdot, t) \to \gamma_\infty (\cdot) \quad \text{as} \quad t \to \infty.
 \end{align*}
Furthermore, by applying the same argument to $\partial_u^l \gamma(\cdot, t)$ for $l \in \N$, we obtain the full limit convergence in the $C^\infty$-topology.

\noindent
{\bf Step 3.}{ \sl We prove that $\gamma_\infty$ is a circular arc and give the rate of convergence.} By Lemma \ref{lem:Phiexpdecay} and the argument in Step 2, we see that $k (\cdot, t)$ converges to a constant as $t \to \infty$. Moreover, as shown in Step 1, $\gamma_\infty$ is an equilibrium of \eqref{MP}, and hence a stationary solution to \eqref{P}. The preservation of the signed area yields $A[\gamma_\infty] = A[\gamma_0] > 0$. Therefore, Corollary~\ref{cor:k>0arc} implies that $\gamma_\infty$ is a circular arc with positive curvature. From the same argument as in \cite[Theorem 5.8]{HO_2025}, the convergence rate follows from \eqref{eq:Gconvergence} and Lemma \ref{lem:Phiexpdecay}. This completes the proof.
\end{proof}
\begin{remark}
The condition $|\omega|<\pi$ in the definition of $\mathcal D$ allows the rotation number to have either sign. This causes no difficulty in the proof of Theorem \ref{thm:convergenceofMP}, since $\omega$ appears in the estimates only through $\omega^2$ or $|\omega|$. Nevertheless, for every initial curve satisfying all the assumptions of Theorem \ref{thm:convergenceofMP}, the positivity of the rotation number follows a posteriori. Indeed, let $k_\infty$ denote the constant curvature of the limiting circular arc $\gamma_\infty$. By the preservation of the rotation number and Corollary \ref{cor:k>0arc}, we have
\[
\omega[\gamma_0]
=
\omega[\gamma_\infty]
=
k_\infty L[\gamma_\infty]
>0.
\]
Thus, in the negative-rotation sector of $\mathcal D$, the positive-area and smallness assumptions of Theorem \ref{thm:convergenceofMP} cannot hold simultaneously.
\end{remark}
\section{Proof of Theorem \ref{shukekka}}\label{sec:pfofshukekka}
We are now in a position to prove Theorem \ref{shukekka}.
\begin{proof}[Proof of Theorem \ref{shukekka}]
By Theorems \ref{thm:local-in-timesolofMP}, \ref{thm:global-in-timesolofMP} and \ref{thm:convergenceofMP}, \eqref{MP} admits a unique smooth global-in-time solution that converges exponentially to a circular arc in the $C^\infty$-topology. Since every solution to \eqref{MP} is also a solution to \eqref{P}, this establishes the existence of a smooth global-in-time solution to \eqref{P}. It remains to prove that the solution is unique up to time-dependent reparametrizations. We use the constant-speed reparametrization defined by the inverse of the normalized arc-length parametrization, following Rupp--Spener \cite[Definition 4.9]{RS_2024}. We also refer to Mantegazza--Pluda--Pozzetta \cite[Proposition 2.16 and Theorem 3.18]{MPP_2021} for the general reparametrization argument underlying geometric uniqueness. For the reader's convenience, we provide the details in the present setting. Let $\gamma: (-1, 1) \times [0, T) \to \R^2$ be a smooth regular solution to \eqref{P} with initial curve $\gamma_0: [-1 ,1] \to \R^2$. Recall that \eqref{A1} gives
\[
|\partial_u \gamma_0| = \frac{L[\gamma_0]}{2} \quad \text{for all} \quad u \in [-1, 1].
\] 
We define the normalized arc-length reparametrization $\rho$ by
\begin{align}\label{eq:rho}
\rho (u, t) := -1 + \frac{2}{L[\gamma(t)]}\int_{-1}^u |\partial_\xi \gamma(\xi, t)|\, d\xi. 
\end{align}
Since $\gamma$ is regular, we have
\[
\partial_u \rho = \frac{2 |\partial_u \gamma|}{L[\gamma]} > 0 \quad \text{for all} \quad (u, t) \in [-1, 1] \times [0, T)
\]
Moreover, 
\[
\rho(\pm1, t) = \pm1 \quad \text{for all} \quad t \in [0, T).
\]
Thus, we see that
\[
\rho(\cdot, t): [-1, 1] \to [-1, 1]
\]
is a smooth orientation-preserving diffeomorphism. Let $\psi(\cdot, t) := \rho(\cdot, t)^{-1}$ and define $\widetilde{\gamma} (u, t) := \gamma (\psi(u, t), t)$. By the definition of $\psi$, we have
\begin{align}\label{eq:rhopsi}
u = \rho (\psi (u, t),  t).
\end{align}
This together with \eqref{eq:rho} implies that
\[
\int_{-1}^{\psi(u, t)} |\partial_\xi \gamma(\xi, t)| \, d\xi = \frac{u + 1}{2} L[\gamma].
\]
Differentiating this equality with respect to u, we obtain
\[
\partial_u \psi(u, t)|\partial_u \gamma(\psi(u, t), t)| = \frac{L[\gamma]}{2}.
\]
Since 
\[
\partial_u \tilde{\gamma} (u, t) = \partial_u \psi (u, t) \partial_u \gamma(\psi (u, t), t),
\]
and length is invariant under reparametrization, it follows that
\begin{align}\label{eq:constantspeed}
|\partial_u \widetilde{\gamma}| = \frac{L[\gamma]}{2} = \frac{L[\widetilde{\gamma}]}{2} \quad \text{for all} \quad t \in [0, T).
\end{align}
Thus, $\widetilde{\gamma} (\cdot, t)$ is the constant-speed reparametrization of $\gamma (\cdot, t)$. We next verify that the reparametrization preserves the parametrized initial curve. Since $\gamma (u, 0) = \gamma_0(u)$, it follows from \eqref{A1} and \eqref{eq:rho} that
\begin{align*}
  \rho(u, 0) = -1 + \frac{2}{L[\gamma_0]} \int_{-1}^u |\partial_\xi \gamma_0 (\xi)| \, d\xi = u.
\end{align*} 
In particular, this together with \eqref{eq:rhopsi} implies that $\psi (u, 0) = u$ for all $u \in [-1, 1]$. Consequently, $\widetilde{\gamma}(u, 0) = \gamma_0(u)$. Moreover, a reparametrization affects only the tangential component of the velocity. Hence, there exists a smooth function $\widetilde{\alpha}:[-1, 1] \times [0, T) \to \mathbb{R}$ such that
\[
\partial_t \widetilde{\gamma} = F_m[\tilde{k}] \widetilde{{\bm \nu}} + \widetilde{\alpha} \widetilde{{\bm \tau}},
\]  
where $\widetilde{k}$, $\widetilde{\bm \nu}$ and $\widetilde{\bm \tau}$ denote the curvature, the unit normal vector, and the unit tangential vector of $\widetilde{\gamma}$, respectively. 
Since $\psi(\pm1, t) = \pm1$ and $\psi (\cdot, t)$ preserves orientation, the conditions on the end points and contact angles of $\widetilde{\gamma}$ are preserved. Under an orientation-preserving reparametrization, the curvature and its arc-length derivatives transform by composition. Since $\psi(\pm1,t)=\pm1$, their vanishing boundary conditions are preserved. Hence, $\widetilde{\gamma}$ satisfies all the boundary conditions in \eqref{P}. Moreover, the constant-speed condition \eqref{eq:constantspeed}, together with the end point conditions
\[
\widetilde{\alpha}(\pm1,t)
=
-F_m[\tilde{k}](\pm1,t)\cot\theta_\pm,
\]
uniquely determines the tangential velocity. Therefore, the same calculation as in Section 3.1 yields
\[
\widetilde{\alpha} = G_m[\tilde{k}].
\]
Consequently, $\widetilde{\gamma}$ is a solution to \eqref{MP} with initial curve $\gamma_0$. We now prove uniqueness. Let $\gamma_1, \gamma_2: (-1, 1) \times [0, T) \to \R^2$ be two smooth regular solutions to \eqref{P} with the initial curve $\gamma_0$. By the preceding argument, there exist smooth families of orientation-preserving diffeomorphisms
\[
\psi_i (\cdot, t): [-1, 1] \to [-1, 1] \quad \text{for} \text \quad i \in \{1, 2\},
\] 
such that $\widetilde{\gamma}_i (u, t) := \gamma_i (\psi_i (u, t), t)$ is a solution to \eqref{MP} with initial curve $\gamma_0$. By the uniqueness of the solution to \eqref{MP}, we have $\widetilde{\gamma}_1 \equiv \widetilde{\gamma}_2$. Define $\chi (u, t) := \psi_2 (\psi_1^{-1} (u, t), t)$. Since both $\psi_1$ and $\psi_2$ are smooth orientation-preserving diffeomorphisms of $[-1, 1]$, we have $\chi (\pm1, t) = \pm1$ and
\[
\partial_u \chi (u, t) = \partial_u \psi_1^{-1}(u, t) \left[ \partial_u \psi_2 \right] (\psi_1^{-1}(u, t), t) > 0.
\]
Therefore, $\chi$ is a smooth orientation-preserving diffeomorphism of $[-1, 1]$. Moreover, since $\psi_i (u, 0) = u$ for $i \in \{1, 2\}$, we have $\chi (u, 0) = u$. Finally, $\widetilde{\gamma}_1 \equiv \widetilde{\gamma}_2$ implies 
\begin{align*}
\gamma_1 (u, t) &= \widetilde{\gamma}_1 (\psi_1^{-1}(u, t), t) = \widetilde{\gamma}_2 (\psi_1^{-1} (u, t), t)\\
&= \gamma_2 (\psi_2 (\psi_1^{-1}(u, t), t), t)= \gamma_2 (\chi (u, t), t).
\end{align*}
Thus, $\gamma_1$ and $\gamma_2$ differ only by a time-dependent reparametrization. Hence, the smooth regular solution to \eqref{P} with initial curve $\gamma_0$ is unique up to time-dependent orientation-preserving reparametrizations fixing the end points.
\end{proof}


\subsection*{Acknowledgment}
The author would like to thank Shinya Okabe for fruitful discussions. The author is also grateful to Yoshihito Kohsaka and Takashi Kagaya for valuable technical advice on proving the local-in-time solvability of the problem \eqref{MP}.

\end{document}